\documentclass[draftcls, onecolumn, 11pt]{IEEEtran}
\usepackage{tipa}

\usepackage{amsfonts}
\usepackage{amsmath}
\usepackage{amssymb}
\usepackage{graphicx}    % epsͼÏñÖ§³Ö
\usepackage{cite}
\usepackage{subfigure}

\usepackage{amsfonts, bm, color}

\newtheorem{theorem}{Theorem}[section]
\newtheorem{lemma}{Lemma}[section]

\newtheorem{assump}{Assumption}[section]
\newtheorem{remark}{Remark}[section]

\newcommand{\trace}{{\rm Tr}}

\newcommand{\st}{{~\rm s.t.}}

\newcommand{\bI}{\mathbf{I}}

\newcommand{\bP}{\mathbf{P}}

\newcommand{\bQ}{\mathbf{Q}}

\newcommand{\bzero}{\bm{0}}

\newcommand{\bA}{\mathbf{A}}

\newcommand{\bB}{\mathbf{B}}

\newcommand{\bE}{\mathbf{E}}
\newcommand{\bU}{\mathbf{U}}

\newcommand{\bM}{\mathbf{M}}
\newcommand{\bN}{\mathbf{N}}

\newcommand{\bX}{\mathbf{X}}

\newcommand{\cX}{\mathcal{X}}
\newcommand{\tX}{\tilde{\bX}}
\newcommand{\Xt}{\tilde{X}}
\newcommand{\bY}{\mathbf{Y}}

\newcommand{\bSigma}{\mathbf{\Sigma}}

\newcommand{\mI}{\mathcal{I}}

\newcommand{\bq}{\bm{q}}

\newcommand{\bx}{\bm{x}}

\newcommand{\tx}{\tilde{\bx}}

\newcommand{\by}{\bm{y}}

\newcommand{\bz}{\bm{z}}
\newcommand{\bb}{\bm{b}}

\newcommand{\bd}{\bm{d}}

\newcommand{\Rdom}{\mathbb{R}}
\newcommand{\expect}{\mathbb{E}}

\newcommand{\bone}{\bm{1}}

\newcommand{\norm}[1]{\left\Vert#1\right\Vert}

\begin{document}

% paper title
\title{Inexact Block Coordinate Descent Methods For Symmetric Nonnegative Matrix Factorization}
%A SDP Approach for Range-free Localization in Wireless Sensor Network

% author names and affiliations
% use a multiple column layout for up to three different
% affiliations
\author{
\authorblockN{Qingjiang Shi, Haoran Sun, Songtao Lu, Mingyi Hong, Meisam Razaviyayn}%, Mingyi Hong, Enbin Song, Yunlong Cai, Weiqiang Xu, Xiqi Gao}%, Weiqiang Xu, \emph{Senior Member, IEEE}, Tsung-Hui Chang, \emph{Member, IEEE}, Yongchao Wang, Enbing Song}
%\authorblockA{$^1$School of Info. Sci. \& Tech., Zhejiang Sci-Tech University, Hangzhou 310018, China}
%\authorblockA{$^2$Nanjing University of Posts and Telecommunications, Nanjing 210003, China}
%\thanks{Copyright (c) 2014 IEEE. Personal use of this material is permitted. However, permission to use this material for any other purposes must be obtained from the IEEE by sending a request to pubs-permissions@ieee.org.}
%\thanks{The work of Q. Shi and W. Xu is supported by the National Nature Science Foundation of
%China under grant 61302076, 61374020, Key Project of Chinese Ministry of Education under grant 212066, Zhejiang Provincial Natural Science Foundation of China under grant LY12F02042, LQ12F01009, LQ13F010008, the Science Foundation of Zhejiang Sci-Tech University (ZSTU) under grant 1203805Y, and The State Key Laboratory of Integrated Services Networks, Xidian University under grant ISN14-08. The work of T.-H. Chang is supported by Ministry of Science and Technology, Taiwan, under Grant NSC 102-2221-E-011-005-MY3. The work of Y. C. Wang is supported by the National Nature Science Foundation of
%China under grant 61372135. The work of E. Song is supported by the National Nature Science Foundation of
%China under grant 61473197.}
\thanks{Q. Shi is now with the Dept. of Industrial and Manufacturing Systems Engineering, Iowa State University, IA 50011, USA. He is also with the School of Info. Sci. \& Tech., Zhejiang Sci-Tech University, Hangzhou 310018, China. Email: qing.j.shi@gmail.com}
%%He is also with The State Key Laboratory of Integrated Services Networks, Xidian University.
%Email: qing.j.shi@gmail.com}
%\thanks {H. Sun and M. Hong are all with the Dept. of Industrial and Manufacturing Systems Engineering, Iowa State University, IA 50011, USA. Email: \{mingyi, hrsun\}@iastate.edu}
%\thanks{S. Lu is with the Dept. of Electrical and Computer Engineering, Iowa State University, IA 50011, USA. Email: %songtao@iastate.edu}
\thanks {H. Sun, S. Lu and M. Hong are all with the Dept. of Industrial and Manufacturing Systems Engineering, Iowa State University, IA 50011, USA. Email: \{mingyi, hrsun, songtao\}@iastate.edu}
\thanks {M. Razaviyayn is with the Dept. of Electrical and Computer Engineering,
Stanford University, Standford 94305, USA. Email: meisamr@stanford.edu.}
%\thanks {Y. Cai is with the Department of Information Science
%and Electronic Engineering, Zhejiang University, Hangzhou 310027, China.
%E-mail: ylcai@zju.edu.cn}
%\thanks{W. Xu is with the School of Info. Sci. \& Tech. Zhejiang Sci-Tech University, Hangzhou 310018, China. Email: wq.xu@126.com}
%%\thanks {T.-H. Chang is with the Department of Electronic and Computer Engineering, National Taiwan University
%%of Science and Technology, Taipei 10607, Taiwan, (R.O.C.). Email: tsunghui.chang@ieee.org.}
%%\thanks {Y. C. Wang is with The State Key Laboratory of Integrated Services Networks, Xidian University. Email: ychwang@mail.xidian.edu.cn}
%\thanks {E. Song is with the College of Mathematics, Sichuan University, Chendu, Sichuan 610064, China. Email: e.b.song@163.com. His work is supported by NSFC }
}

%\thanks{This work is supported by National Nature Science Foundation of
%P.R.China under Grant No.60772100 and Science and Technology
%Committee of Shanghai Municipality under Grant No.05DZ15004.}

\maketitle

\begin{abstract}
Symmetric nonnegative matrix factorization (SNMF) is equivalent to computing a symmetric
nonnegative low rank approximation of a data similarity matrix. It inherits the good data interpretability of the well-known nonnegative matrix factorization technique and have better ability of clustering nonlinearly separable data. In this paper, we focus on the algorithmic aspect of the SNMF problem and propose simple inexact block coordinate decent methods to address the problem, leading to both serial and parallel  algorithms. The proposed algorithms have guaranteed stationary convergence and can efficiently handle large-scale  and/or sparse SNMF problems. Extensive simulations verify the effectiveness of the proposed algorithms compared to recent state-of-the-art algorithms.
\end{abstract}
\begin{keywords}
Symmetric nonnegative matrix factorization, block coordinate decent, block successive upper-bounding minimization, parallel algorithm, stationary convergence.
\end{keywords}

\IEEEpeerreviewmaketitle

\section{Introduction}
Clustering is the task of grouping a set of data points into different clusters according to some measure of data similarity. It is a common technique for statistical data analysis and has been widely used in the fields such as machine learning, pattern recognition and data compression, etc.. In some applications, clustering is performed on the data which has inherent nonnegativity. This motivates the great interests in the application of nonnegative matrix factorization (NMF) to clustering. NMF has been shown to be very effective for clustering linearly separable data because of its ability to automatically extract sparse and easily interpretable factors\cite{Lee1999}. As a close relative (or a variant) of NMF, symmetric NMF (SNMF) is more directly related to the clustering problems: it can be even viewed as a relaxed version of two classical clustering methods: $K$-means clustering and spectral clustering\cite{Kuang2012}. It inherits the good interpretability of NMF and has received considerable interests due to its better ability of clustering nonlinearly separable data given a data similarity matrix\cite{Kuang2012,Gill2015,Ding2005,Long2005,Long2007,Huang2014,He2011}. This paper focuses on the algorithmic aspect of SNMF.

The basic SNMF problem is described as follows. Given $n$ data points, a similarity matrix $\bM\in \Rdom^{n\times n}$ can be generated by evaluating the similarity between any two data points in terms of an appropriate similarity metric. The SNMF problem of determining $r$ clusters is to find a low-rank nonnegative factorization matrix $\bX\in \Rdom^{n\times r}$ (with $r\ll n$ generally) such that $\bM\approx\bX\bX^T$. Using Frobenius norm, the basic SNMF problem is often formulated as follows
\begin{equation}\label{eq:SNMF0}
\min_{\bX\geq 0}F(\bX)\triangleq\Vert\bM-\bX\bX^T\Vert^2
\end{equation}
where the nonnegativity constraint $\bX\geq 0$ is a componentwise inequality. Such model often leads to sparse factors which are better interpretable\footnote{With the requirement of nonnegativity on $\bX$, the index of the largest entry of the $i$-th row of $\bX$ can be simply thought of as the cluster label of the $i$-th data points. This greatly facilitates clustering once the symmetric factorization of $\bM$ is obtained.} for problems such as image or document analysis. Problem \eqref{eq:SNMF0} is equivalent to the so-called \emph{completely positive matrix} problem, which postulates whether $\bM$ can be exactly factorized as $\bM=\bX\bX^T$ with $\bX\geq 0$. This problem is originated in inequality theory and quadratic forms\cite{Diana1962,Hall1963} and plays an important role in discrete optimization\cite{Hall1962}. Note that the problem---whether a matrix is completely positive -- has been recently shown to be NP-hard\cite{Hall2013}. As such, problem \eqref{eq:SNMF0} is generally NP-hard and thus efficient numerical algorithms are desired to obtain some high-quality, but not necessarily globally optimal,  solutions.

So far, there is limited algorithmic works on SNMF as compared to the general NMF. Following the popular multiplicative update rule proposed for the NMF problem\cite{Lee1999}, some pioneering works\cite{Ding2005,Long2005,Long2007} on SNMF have proposed various modified multiplicative update rules for SNMF. The multiplicative update rules are usually simple to implement but require the similarity matrix $\bM$ to be at lease nonnegative to ensure the nonnegativity of $\bX$ at each iteration. For the case of \emph{positive definite} and \emph{completely positive} similarity matrix $\bM$, the authors in \cite{He2011} developed three faster parallel multiplicative update algorithms, including the basic multiplicative update algorithm, $\alpha$-SNMF algorithm, and $\beta$-SNMF algorithm, all converging to stationary solutions to the SNMF problem. It was numerically shown that the latter two outperform the first one and other previous multiplicative-update-based SNMF algorithms. It should be stressed that all the above multiplicative-update-based SNMF algorithms \emph{implicitly} assume \emph{positive} $\bX$ at each iteration to derive the corresponding auxiliary functions\cite{Ding2005,Long2005,Long2007,He2011} and require arithmetic division operators in the updates. Consequently, they may not be numerically stable when some entries of $\bX$ reach zero during iterations and cannot even guarantee stationary convergence in some cases.

While the above algorithms require the similarity matrix $\bM$ to be nonnegative, it is noted that the similarity matrix $\bM$ in graph-based clustering may be neither nonnegative nor positive definite (e.g., the matrix $\bM$ could have negative entries in some definition of kernel matrices
used in semi-supervised clustering \cite{Kulis2005}). Hence, the aforementioned algorithms do not apply to the SNMF problem with a general similarity matrix $\bM$. In \cite{Kuang2012}, the authors proposed two numerical optimization methods to find stationary solutions to the general SNMF problem. The first one is a Newton-like algorithm which uses partial second-order information (e.g., reduced Hessian or its variants) to reduce the computational complexity for determining the search direction, while the second one is called alternating nonnegative least square (ANLS) algorithm, where the SNMF problem is split into two classes of nonnegative least-square subproblems by using penalty method coupled with two-block coordinate descent method. Technically, the ANLS algorithm is a penalty method whose  stationary convergence is much impacted by the choice of the penalty parameter. As a result, the ANLS algorithm often comes up with an approximate solution which is generally less accurate than the solution provided by the Newton-like algorithm.
%$\bX$ into two variables with an additional equality constraint and then penalizes the violation of the equality constraint into the objective with the resultant problem solved by using two-block coordinate descent method, leading to two classes of nonnegative least square subproblems.
%The Newton-like algorithm can reach stationary solutions but the stationary convergence of the ANLS algorithm could be much impacted by the choice of the penalty parameter. Generally, the ANLS algorithm cannot reach stationary solution when a large but fixed penalty parameter is used.

If we view each entry of $\bX$ or a set of entries of $\bX$ as one block variable, the SNMF problem is a multi-block optimization problem. A popular approach for solving multi-block optimization problems is the block coordinate descent (BCD) approach\cite{Bertsekas_book,Wright2015}, where only one of the block variables is optimized at each iteration while the remaining blocks are kept fixed. When the one-block subproblem can be easily solved, the BCD algorithms often perform very efficiently in practice. Furthermore, since only one block is updated at each iteration, the per-iteration
storage and computational burden of the BCD algorithm could be very low, which is desirable for solving large-scale problems. Due to the above merits, the BCD method has been recently used to solve the SNMF problem\cite{Gill2015}, where each entry of $\bX$ is viewed as one block variable and the corresponding subproblem is equivalent to finding the best root of a cubic equation. The proposed BCD algorithm in \cite{Gill2015} was designed in an efficient way by exploiting the structure of the SNMF problem. However, since the subproblem of updating each entry of $\bX$ may have multiple solutions, the existing convergence results cannot be applied to the proposed BCD algorithm in \cite{Gill2015}. Though much effort was made in \cite{Gill2015} to prove the convergence of the proposed algorithm to stationary solutions, we find that there exists a gap in the convergence proof of Theorem 1 in \cite{Gill2015} (see Sec. II-A.2)).

To overcome the weakness of the BCD algorithm and make it more flexible, an inexact BCD algorithm, named block successive upper-bounding minimization (BSUM)  algorithm, was proposed for multi-block minimization problems\cite{Razav2013}, where the block variables are updated by successively minimizing
a sequence of approximations of the objective function which are either locally tight upper bounds of the objective function or strictly convex local approximations of the objective function. Unlike the classical BCD method which requires for convergence the uniqueness of the minimizer of every one-block subproblem, the BSUM algorithm  and its variants or generalizations are guaranteed to achieve convergence to stationary solutions under very mild conditions\cite{Razav2013}. Considering these merits of the BSUM algorithm, this paper proposes BSUM-based algorithms to address (regularized) SNMF problems.

The contribution of this paper is threefold:% stop anytime/parameter-free/improve at every iteration
\begin{itemize}
\item First, two cyclic BSUM algorithms are proposed for the SNMF problem, including an entry-wise BSUM algorithm and a vector-wise BSUM algorithm.  The key to both algorithms is to find proper upper bounds for the objective functions of the subproblems involved in the two algorithms.%In the former, one entry is updated in each time, while in the latter, a row of entries are simultaneously updated in each time.
\item Then, we study the convergence of permutation-based randomized BSUM (PR-BSUM) algorithm. For the first time, we prove that this algorithm can monotonically converge to the set of stationary solutions. As a result, this paper settles the convergence issue of the exact BCD algorithm\cite{Gill2015} using permuted block selection rule.
    %The PR-BSUM algorithm applied to the SNMF problem could often achieve faster convergence than the cyclic BSUM algorithms for large-scale SNMF problems.
\item Lastly, we propose parallel BSUM algorithms to address the SNMF problem. By distributing the computational load to multiple cores of a high-performance processor, the parallel BSUM algorithms can efficiently handle large-scale SNMF problems.
\end{itemize}

The remainder of this paper is organized as follows. We present two cyclic inexact BCD methods, termed as sBSUM and vBSUM, for the SNMF problem in Section II. In Section III, we study the permutation-based randomized BSUM algorithm in a general framework and its application to the SNMF problem. In Section IV, we discuss how to implement the sBSUM algorithm and vBSUM algorithm in parallel in a multi-core/processor architecture. Finally, Section V presents some simulation results, followed by a conclusion drawn in Section VI.

\emph{Notations}: Throughout this paper, we use uppercase bold letters for matrices,
lowercase bold letters for column vectors, and regular letters for scalars. For a matrix $\bX$, $\bX^T$ and $\bX^\dag$ denote the transpose and pseudo-inverse of $\bX$, respectively, and $\bX\geq0$ means entry-wise inequality. $\norm{\bX}$ and $\trace(\bX)$ denotes the Frobenius norm and trace of $\bX$, respectively. We use the notation $X_{ij}$ (or $(\bX)_{ij}$), $\bX_{i:}$ and $\bX_{:j}$ to refer to the $(i,j)$-th entry, the $i$-th row and the $j$-th column of $\bX$, respectively, while $\bX_{a:b,:}$ the submatrix of $\bX$ consisting of elements from the $a$-th row through $b$-th row. $\bI$ denotes the identity matrix whose size will be clear from the context.  We define an element-wise operator $[\bx]_{+}=\max(\bx,0)$. For a function $f(x)$, $f'(x)$ (or $\nabla f(\bx)$) denotes its derivative (or gradient) with respect to the variable $x$.

\section{Cyclic BSUM Algorithms For SNMF}
This paper aims to provide algorithmic design in a broad view for the following generalized formulation of the SNMF problem
\begin{equation}\label{eq:SNMF-g}
\min_{\bX\geq 0} \Vert \bM-\bX\bX^T\Vert^2+\rho R(\bX)
\end{equation}
where $\bX$ and $\bM$ are respectively $n$-by-$r$ and $n$ by $n$ matrices, $\rho$ is a scalar, and $R(\bX)$ is a matrix polynomial of up to fourth-order. Problem \eqref{eq:SNMF-g} is in essence a class of multivariate quartic polynomial optimization problem with nonnegative constraints. Besides the basic SNMF problem, this problem arises also from for example:
\begin{itemize}
\item when $\bM$ is a similarity matrix and $R(\bX)=\sum_{i=1}^r\sum_{j=1}^n |X_{ij}|$ or $\sum_{j=1}^n(\sum_{i=1}^r |X_{ij}|)^2$, problem \eqref{eq:SNMF-g} is a regularized SNMF problem which can introduce some degree of sparsity to $\bX$ or the rows of $\bX$\cite{Kuang2012}.
\item when $\bM$ is the covariance matrix of sensor observations of some distributed sources in a sensor network and $R(\bX)=\sum_{i=1}^r\sum_{j=1}^n |X_{ij}|$,  problem \eqref{eq:SNMF-g} models an informative-sensor identification problem with sparsity-awareness and \emph{nonnegative} source signal signature basis\cite{Schizas2013}.
\end{itemize}

It is not difficult to see that, problem \eqref{eq:SNMF-g} is as hard as the basic SNMF problem when $R(\bX)$ is up to fourth-order,  and the main difficulty lies in the fourth-order matrix polynomial, which makes the problem hard to solve. Hence, we focus on the basic SNMF problem throughout the rest of the paper and develop various BSUM-type algorithms for the basic SNMF problem. Without loss of generality, we assume that the matrix $\bM$ is symmetric\footnote{When $\bM$ is not symmetric, the corresponding problem is equivalent to \eqref{eq:SNMF-g} with $\bM$ thereof replaced by $\frac{\bM+\bM^T}{2}$.}. We stress that all the techniques developed below can be easily generalized to the regularized SNMF problem \eqref{eq:SNMF-g}.

\subsection{Scalar-wise BSUM algorithm}
The basic SNMF problem is given by
\begin{equation}\label{eq:SNMF}
\min_{\bX\geq 0} F(\bX)\triangleq\Vert \bM-\bX\bX^T\Vert^2.
\end{equation}
It is readily seen that, when all the entries of $\bX$ but one are fixed, problem \eqref{eq:SNMF} reduces to minimizing a univariate quartic polynomial (UQP), which can be easily handled. Based on this fact, exact BCD algorithms---cyclicly updating each entry of $\bX$ by solving a UQP problem---has been developed for the SNMF problem\cite{Gill2015,Schizas2013}. However, since the solution to the UQP problem may not be unique, the existing convergence theory of BCD does not apply to the exact BCD algorithm developed for the SNMF problem, though much efforts have been made to resolve the convergence issue in \cite{Gill2015,Schizas2013}. In the following, we develop a BSUM-based algorithm for the SNMF problem, where each time we update one entry of $\bX$ by minimizing a locally tight upper bound of $F(\bX)$, hence termed as scalar-BSUM (sBSUM).

\subsubsection{Algorithm design and complexity analysis}In the sBSUM algorithm, each time we optimize the $(i,j)$-th entry of $\bX$ while fixing the others.
For notational convenience, let $x$ denote the $(i,j)$-th entry of $\bX$ to be optimized. Given the current $\bX$, denoted as $\tX$, we can write the new $\bX$ after updating the $(i,j)$-th entry as
$$\bX=\tX+(x-\Xt_{ij})\bE_{ij},$$
where $\bE_{ij}\in \Rdom^{n\times r}$ is a matrix with $1$ in the $(i,j)$-th entry and $0$ elsewhere. Accordingly, we can express the objective function as follows
\begin{equation}\label{eq:FX}
F(\bX)=g(x)+F(\tX),
\end{equation}
where
\begin{equation}\label{eq:gx}
\begin{split}
g(x)&\triangleq%\left\Vert\left(\tX+(x-\tilde{X}_{ij})\bE_{ij}\right)^T\left(\tX+(x-\tilde{X}_{ij})\bE_{ij}\right)-\bM\right\Vert^2\\
%& =
\frac{a}{4}(x-\Xt_{ij})^4+\frac{b}{3}(x-\Xt_{ij})^3+\frac{c}{2}(x-\Xt_{ij})^2\\
&~~~~~+d(x-\Xt_{ij})
\end{split}
\end{equation}
with the tuple $(a,b,c,d)$ given by
\begin{align}
a&=4,\\
b&=12\Xt_{ij},\\
c&=4\left((\tX\tX^T)_{ii}-M_{ii}+(\tX^T\tX)_{jj}+\Xt_{ij}^2\right),\label{eq:c}\\
d&=4\left((\tX\tX^T-\bM)\tX\right)_{ij}\label{eq:d}.%= \left(4\tX\tX^T\tX-2\tX(\bM^T+\bM)\right)_{ij}
\end{align}
The derivation of $(a,b,c,d)$ is relegated to Appendix A. To get an upper bound of $g(x)$, let us define
$$\tilde{g}(x)\triangleq g(x)+\frac{1}{2}\tilde{c}(x-\Xt_{ij})^2,$$
where $\tilde{c}\triangleq\max\left(\frac{b^2}{3a}-c,0\right)$. Clearly, $\tilde{g}(x)$ is  a locally tight upper bound of $g(x)$ at $\Xt_{ij}$. Using this upper bound function, the sBSUM algorithm updates $x$ (i.e., the $(i,j)$-th entry of $\bX$) by solving
\begin{equation}\label{eq:ming}
\min\; \tilde{g}(x)\quad \mbox{s.t.} \; x\ge 0.
\end{equation}

The following lemma states that $\tilde{g}(x)$ is convex and moreover problem \eqref{eq:ming} has a unique solution.
\begin{lemma}\label{lem:lem1}
$\tilde{g}(x)$ is a convex function and a locally tight upper bound of $g(x)$. Moreover, problem \eqref{eq:ming} has a unique solution.
\end{lemma}
\begin{proof}
Clearly, $\tilde{g}(x)$ is a locally tight upper bound of $g(x)$ since $\tilde{g}(\tX_{ij})=g(\tX_{ij})$ and $\tilde{g}(x)\geq g(x)$, $\forall x\geq 0$. Hence, it suffices to show that $g(x)$ is a convex function. By noting that
\begin{align}
%\tilde{g}'(x)&=a(x-\Xt_{ij})^3+b(x-\Xt_{ij})^2+c(x-\Xt_{ij})+d\\
\tilde{g}''(x)&=3a(x-\Xt_{ij})^2+2b(x-\Xt_{ij})+c+\tilde{c},\nonumber\\
&=3a\left(x-\Xt_{ij}+\frac{b}{3a}\right)^2+\max\left(0, c-\frac{b^2}{3a}\right)\geq 0,\nonumber
\end{align}
we infer that $\tilde{g}(x)$ is a convex function. Particularly, $\tilde{g}(x)$ is strictly convex when $c>\frac{b^2}{3a}$. Thus, problem \eqref{eq:ming} has a unique solution when $c>\frac{b^2}{3a}$. For the case when $\frac{b^2}{3a}\geq c$, we have
\begin{align}
\tilde{g}'(x)&=a(x-\Xt_{ij})^3+b(x-\Xt_{ij})^2+\frac{b^2}{3a}(x-\Xt_{ij})+d\nonumber\\
&=a\left(x-\Xt_{ij}+\frac{b}{3a}\right)^3+d-\frac{b^3}{27a^2},\nonumber
\end{align}
where we have used the identity $(u+v)^3=u^3+v^3+3u^2v+3v^2u$ in the second equality. Equating the derivative $\tilde{g}'(x)$ to zero and taking the nonnegative constraint into consideration, we obtain the unique solution to
problem \eqref{eq:ming} as $\max\left(\sqrt[\leftroot{-1}\uproot{6}3]{\frac{b^3}{27a^3}-\frac{d}{a}}+\Xt_{ij}-\frac{b}{3a},0\right)$.
This completes the proof.
\end{proof}

As a result of Lemma \ref{lem:lem1}, the unique solution to problem \eqref{eq:ming} can be found by checking the critical point of $\tilde{g}(x)$ (i.e., the point $x$ where $\tilde{g}'(x)=0$). That is, if the critical point of $\tilde{g}(x)$ is nonnegative, then it is the unique optimum solution to problem \eqref{eq:ming}; otherwise the optimal solution is $x=0$. Hence, the key to solving problem \eqref{eq:ming} is to find the root of the cubic equation $\tilde{g}'(x)=0$ whose coefficients are specified by the tuple $(4, b, \max(c,\frac{b^2}{3a}), d)$. Fortunately, we can easily solve the cubic equation $\tilde{g}'(x)=0$ and obtain the unique closed-form solution to problem \eqref{eq:ming} in the following lemma.
\begin{lemma}
Let $(a, b, c, d)$ be given in Eqs. (6-9) and define
$$p\triangleq \frac{3ac-b^2}{3a^2},~ q\triangleq\frac{9abc-27a^2d-2b^3}{27a^3}, ~ \Delta\triangleq\frac{q^2}{4}+\frac{p^3}{27}.$$
 The unique solution to problem \eqref{eq:ming} can be expressed as
 \begin{equation}\label{eq:cubic_sol}
 X_{ij}=[w]_{+}
 \end{equation}
where $w$ is given by
\begin{equation}\label{eq:cubic_sol2}
w=\left\{\begin{split}&\sqrt[\leftroot{-1}\uproot{9}3]{\frac{q}{2}-\sqrt{\Delta}}+\sqrt[\leftroot{-1}\uproot{9}3]{\frac{q}{2}+\sqrt{\Delta}}, ~~~\textrm{if} ~c>\frac{b^2}{3a}\\
&\sqrt[\leftroot{-1}\uproot{6}3]{\frac{b^3}{27a^3}-\frac{d}{a}}, ~~~~~~~~~~~~~~~~~~\textrm{otherwise}\end{split}\right.
\end{equation}
%where
%$$w=\sqrt[\leftroot{-1}\uproot{9}3]{\frac{q}{2}-\sqrt{\Delta}}+\sqrt[\leftroot{-1}\uproot{9}3]{\frac{q}{2}+\sqrt{\Delta}}$$
%$$w=\sqrt[\leftroot{1}\uproot{11}3]{\frac{1}{2}\left(q+\sqrt{q^2+\frac{4}{27}p^3}\right)}$$
\begin{proof}
As shown in Lemma \ref{lem:lem1} and using the fact $\Xt_{ij}=\frac{b}{3a}$, it is trivial to obtain the unique solution to problem \eqref{eq:ming} when $c\leq\frac{b^2}{3a}$. Thus, it suffices to consider the case when $c>\frac{b^2}{3a}$. The equation $\tilde{g}'(x)=0$ is equivalent to
\begin{align}
\tilde{g}(x)&=a(x-\Xt_{ij})^3+b(x-\Xt_{ij})^2+\frac{b^2}{3a}(x-\Xt_{ij})\nonumber\\
&~~~~~~~~~~~~~~+\left(c-\frac{b^2}{3a}\right)(x-\Xt_{ij})+d\nonumber\\
&=a\left(x-\Xt_{ij}+\frac{b}{3a}\right)^3+\left(c-\frac{b^2}{3a}\right)(x-\Xt_{ij}+\frac{b}{3a})\nonumber\\
&~~~~~~~~~~~~~~-\left(c-\frac{b^2}{3a}\right)\frac{b}{3a}+d-\frac{b^3}{27a^2}\nonumber\\
&=a(x^3+px-q)=0.\nonumber
\end{align}
where we have used the identity $(u+v)^3=u^3+v^3+3u^2v+3v^2u$ in the second equality, and the fact $\Xt_{ij}=\frac{b}{3a}$ and the definitions of $(p,q)$ in the last equality.
Clearly, we have $\Delta=\frac{q^2}{4}+\frac{p^3}{27}>0$ when $c>\frac{b^2}{3a}$. Thus, it follows from Cardano's method \cite{cubic} that the equation $\tilde{g}'(x)=0$ has a unique solution as follows
$$x=\sqrt[\leftroot{-1}\uproot{9}3]{\frac{q}{2}-\sqrt{\Delta}}+\sqrt[\leftroot{-1}\uproot{9}3]{\frac{q}{2}+\sqrt{\Delta}}.$$
Furthermore, acounting for the nonnegative constraint, we can derive the closed-form solution  for the case when $c>\frac{b^2}{3a}$ as shown in \eqref{eq:cubic_sol}. This completes the proof.
\end{proof}
\end{lemma}

\begin{table}
\centering
\caption{Algorithm 1: sBSUM algorithm for SNMF}
\begin{tabular}{|p{3.5in}|}
\hline
\begin{itemize}
%\item \textbf{Input}: $\{\bh_{jk}\}_{j=1}^K$, $g_k(\bm{\mI}_k^r)$, $\{\mI_{jk}\}_{j\neq k}$
%\item \textbf{output}: $\{\nu_{jk}\}_{j=1}^K$
%\hline\\
\item [0.] initialize $\bX$, calculate $\bX^T\bX$ and $(\bX^T\bX)_{ii}$, $i=1,2,\ldots, n$% $\tQ=\bQ$, $\tS=\bS$, and and $\bR=\bH_{RR}^H\bQ$.
%\item [1.]\; calculate $\bX$
\item [1.]\; \textbf{repeat}
\item [2.] \; \quad\quad \textbf{for} each $i\in\{1,2\ldots,n\}$ and $j\in\{1,2\ldots,r\}$
\item [3.]\;\quad\quad\quad $b=12\bX_{ij}$
\item [4.]\;\quad\quad\quad  $c=4\left((\bX\bX^T)_{ii}-\bM_{ii}+(\bX^T\bX)_{jj}+\bX_{ij}^2\right)$
%\item [4.]\;\quad\quad\quad  $c=4\left(\norm{(\bX)_{i:}}^2-\bM_{ii}+(\bX^T\bX)_{jj}+\bX_{ij}^2\right)$
\item [5.]\;\quad\quad\quad  $d=4\bX_{i:}(\bX^T\bX)_{:j}-4\bM_{i:}\bX_{:j}$
\item [6.]\; \quad\quad\quad $x =$solve$(4, b, \max(c,\frac{b^2}{3a}), d)$ %\/\/ find the unique root of the cubic equation
\item [7.]\; \quad\quad\quad  $(\bX\bX^T)_{ii}=(\bX\bX^T)_{ii}+x^2-X_{ij}^2$%2(x-\bX_{ij})\bX_{ij}+(x-\bX_{ij})^2)$
\item [8.]\; \quad\quad \quad $(\bX^T\bX)_{j:}=(\bX^T\bX)_{j:}+(x-X_{ij})\bX_{i:}$
\item [9.]\; \quad\quad \quad $(\bX^T\bX)_{:j}=(\bX^T\bX)_{j:}^T$%(\bX^T\bX)_{:j}+(x-\bX_{ij})\bX_{i:}^T$
\item[10.]\; \quad\quad \quad $(\bX^T\bX)_{jj}=(\bX^T\bX)_{jj}+(x-X_{ij})^2$
\item[11.]\; \quad\quad \quad $X_{ij}=x$
\item[12.]\;  \quad\quad \textbf{end}
\item [13.]\; \textbf{until} some termination criterion is met
%\item [7.]\; each transmitter sends $\rho_k$ to its receiver
\end{itemize}
\\
\hline
\end{tabular}%\vspace{-7pt}
\end{table}

%Let us consider problem the unconstrained problem
%\begin{equation}\label{eq:MF}
%\min_{\bX} g(\bX)=\Phi(\bX\bX^T-\bM)
%\end{equation}
%
%\begin{theorem}
%Suppose that $\Phi(\bZ)$ is a convex function. Let $\barX$ be a stationary solution of problem \eqref{eq:MF}.  Suppose $\barX$ is of full column rank, then $\barX$ is a globally optimum solution to problem \eqref{eq:MF}.% \eqref{eq:SNMF}.
%\end{theorem}
%\begin{proof}
%%\begin{align}
%%&g(\bX)-g(\barX)\\
%%&\geq \left<\bX\bX^T-\barX\barX^T, \nabla\Phi(\barX\barX^T-\bM)\right>\\
%%&\geq  \left<(\bX\bX^T-\barX\barX^T)\barX(\barX^T\barX)^{-1}, 2(\barX\barX^T-\bM)\barX\right>\\
%%&=0, \forall \bX, \forall \barX
%%\end{align}
%\begin{align}
%&f(\bX)-f(\barX)\\
%&\geq \left<\bX\bX^T-\barX\barX^T, 2(\barX\barX^T-\bM)\right>\\
%&\geq  \left<(\bX\bX^T-\barX\barX^T)\barX(\barX^T\barX)^{-1}, 2(\barX\barX^T-\bM)\barX\right>\\
%&=0, \forall \bX, \forall \barX
%\end{align}
%\end{proof}
%\subsection{The complexity analysis}
To summarize, the basic sBSUM algorithm proceeds as follows. Each time we pick $X_{ij}$ to be optimized---first calculate $(a,b,c,d)$ and then update $X_{ij}$ according to \eqref{eq:cubic_sol}. It is readily seen that computing $c$ in \eqref{eq:c} and $d$ in \eqref{eq:d} are the most computationly costly steps of the sBSUM algorithm. In order to save memory and computational overhead, we can use the following expression
\begin{align}
d&=4\left(\tX(\tX^T\tX)-\bM\tX\right)_{ij}\nonumber\\
&=4\tX_{i:}(\tX^T\tX)_{:j}-4\bM_{i:}\tX_{:j}\label{eq:dupdate}
\end{align}
to compute $d$ as the matrix $\bX^T\bX$ has a smaller size than  the matrix $\bX\bX^T$. Moreover, we have
\begin{align}
&\bX^T\bX=\left(\tX+(x-\Xt_{ij})\bE_{i,j}\right)^T\left(\tX+(x-\Xt_{ij})\bE_{ij}\right)\nonumber\\
&=\tX^T\tX+(x-\Xt_{ij})\tX^T\bE_{ij}+(x-\Xt_{ij})\bE_{ij}^T\tX\nonumber\\
&~~~~~~~~~~~~~~~~~~~~~~~~~+(x-\Xt_{ij})^2\bE_{ij}^T\bE_{ij}\label{eq:XtX}\\
&=\tX^T\tX{+}(x{-}\Xt_{ij})\tX^T\bE_{ij}{+}(x{-}\Xt_{ij})\bE_{ij}^T\tX{+}(x{-}\Xt_{ij})^2\bE_{jj}\nonumber
\end{align}
where $\bE_{jj}\in \Rdom^{r\times r}$ is a matrix with $1$ in the $(j,j)$-th entry and $0$ elsewhere. Note that $\tX^T\bE_{ij}$ is a null matrix with its $j$-th column being $\Xt_{i:}^T$ while $\bE_{ij}^T\tX$ is a null matrix with its $j$-th row being $\Xt_{i:}$. Hence, we only need to update the matrix $(\bX^T\bX)$'s $j$-th row, $j$-th column, and $(j,j)$-th entry once the $(i,j)$-th entry of $\bX$ is updated, which can be done recusively.

Computing $c$ in \eqref{eq:c} requires the knowledge of $(\bX\bX^T)_{ii}$. Similarly, we have
\begin{align}
&\bX\bX^T=\left(\tX+(x-\Xt_{ij})\bE_{ij}\right)\left(\tX+(x-\Xt_{ij})\bE_{ij}\right)^T\nonumber\\
&=\tX\tX^T+(x-\Xt_{ij})\bE_{ij}\tX^T+(x-\Xt_{ij})\tX\bE_{ij}^T\nonumber\\
&~~~~~~~~~~~~~~~~~~+(x-\Xt_{ij})^2\bE_{ij}\bE_{ij}^T\nonumber\\
&=\tX\tX^T{+}(x{-}\Xt_{ij})\bE_{ij}\tX^T{+}(x-\Xt_{ij})\tX\bE_{ij}^T{+}(x-\Xt_{ij})^2\bE_{ii}\nonumber
\end{align}
where $\bE_{ii}\in \Rdom^{n\times n}$ is a matrix with $1$ in the $(i,i)$-th entry and $0$ elsewhere.
It follows that
\begin{equation}\label{eq:update_xtx}
\begin{split}
&(\bX\bX^T)_{kk} \\
&= \left\{\begin{split} &(\tX\tX^T)_{kk}+2(x-\Xt_{ij})\Xt_{ij} +(x-\Xt_{ij})^2~~\textrm{if}~k= i\\
&(\tX\tX^T)_{kk} ~~~~~~~~~~~~~~~~~~~~~~~~~~~~~~~~~~~~~~~\textrm{otherwise}
\end{split}\right.
\end{split}
\end{equation}
Hence, we only need to recursively update $(\bX\bX^T)_{ii}$ once the $(i,j)$-th entry of $\bX$ is updated.

Using \eqref{eq:dupdate} for computing $d$, and \eqref{eq:XtX} and \eqref{eq:update_xtx} for recursive update, we formally present the  proposed sBSUM algorithm in TABLE I, where the function $solve(a_3, a_2,a_1, a_0)$ in Step 6 is to find the unique solution to a cubic equation whose coefficients are specified by the tuple $(a_3, a_2,a_1, a_0)$, followed by a projection onto the nonnegative orthant (i.e., Eq. \eqref{eq:cubic_sol} for Step 6); Steps 7-10 are recursive updates to reduce computational burden. According to the table and the previous analysis, we only need to store the matrix $(\bX^T\bX)$, $\bX$, $\bM$, and the matrix $(\bX\bX^T)$'s diagonal entries in the sBSUM algorithm, which require $O(r^2)$, $O(nr)$, $O(K)$ and $O(n)$ space in memory, respectively. Here, $K$ denotes the number of non-zero entries of $\bM$, which is generally $O(n^2)$ in the dense matrix case while $O(n)$ in the sparse matrix case. Hence, the sBSUM algorithm requires $O(\max(K, nr))$ space in memory. In addition, it is seen that the most costly step of the sBSUM is computing $d$, that is, the computation of $\bM\bX$ dominates the per-iteration computational complexity of the sBSUM algorithm (we refer to updating all entries of $\bX$ as one iteration) . Thus, it can be shown that the per-iteration computational complexity of the sBSUM algorithm is $O(rK)$ in the sparse case while $O(n^2r)$ in the dense case.

\subsubsection{Comparison of exact (cyclic) BCD\cite{Gill2015} and (cyclic) sBSUM}
The exact BCD algorithm and the sBSUM have similar iterations (cf. Algorithm 4 in \cite{Gill2015} and our Algorithm 1) and the same complexity. The main difference between them lies in Step 6:
% which updates the $(i,j)$-th entry of $\bX$.
if we replace Step 6 with $x =solve(4, b, c, d)$ and let the `solve' function denote finding the best nonnegative root (to the corresponding cubic equation) among up to three real roots, the sBSUM  reduces to the exact BCD algorithm. Thus, when $c>\frac{b^2}{3a}$ holds for all iterations, the sBSUM  is exactly the same as the exact BCD algorithm. Clearly, a notable advantage of the sBSUM  is that a closed-form unique real root is guaranteed in Step 6, but this is not the case for the BCD algorithm (Algorithm 1 in \cite{Gill2015} is invoked to find the best root). Thanks to such uniqueness of the subproblems' solutions, the sBSUM is guaranteed to converge to stationary solutions;  see the argument in \cite{Razav2013}. On the other hand, convergence is {\it not} guaranteed for multi-block  cyclic BCD algorithm if each of its subproblem has possibly multiple solutions; see a well-known example due to Powell \cite{PowellBCD73}. The latter result implies that the cyclic BCD algorithm, although  having similar iterations and the same complexity as sBSUM, may not be a good choice for  the SNMF problem. %Moreover, a classical convergence result for nonconvex objective functions requires the solution of each subproblem in the BCD algorithm to be unique\cite{Bertsekas_book}. In sum, we cannot expect a general convergence result for the cyclic BCD applied to nonconvex functions\cite{Wright2015,PowellBCD73}.

It is worth mentioning that, the work \cite{Gill2015} attempted to provide a convergence proof for their cyclic-BCD-based SNMF algorithm. However, there is a gap in their proof as explained as follows. Let $\{\bX_{(t)}\}_{t=0}^{\infty}$ be the sequence generated by the BCD algorithm and $\{\bX_{(t_k)}\}_{k=0}^{\infty}$ be one of its subsequences that converges to $\bar{\bX}$. In the proof of Theorem 1 in \cite{Gill2015}, the authors start by arguing that there exists $(i,j)$ and $p$ such that $\bX+p\bE_{ij}\geq 0$ and
\begin{equation}\label{eq:contadiction}
F(\bar{\bX}+p\bE_{ij})= F(\bar{\bX})-\epsilon<f(\bar{\bX})
\end{equation}
for some $\epsilon>0$ if $\bar{\bX}$ is for contrary assumed to be not a stationary solution, and finally obtain
\begin{equation}\label{eq:sufficient_decrease}
F(\bX_{(t_{k+1})})\leq F(\bX_{(t_{k})})-\epsilon, \forall k>\bar{K}
\end{equation}
where $\bar{K}$ is a sufficiently large integer. Eq. \eqref{eq:sufficient_decrease} means sufficient decrease in the objective function values, resulting in the unboundedness of $F$, i.e., a contradiction, implying that $\bar{\bX}$ is a stationary solution (because $F$ is lower bounded). However, there is a gap in their derivation: in fact the scalar $\epsilon$ in \eqref{eq:contadiction} should be related to $(i,j)$ or the iteration. As a result, the scalar $\epsilon$ in Eq. \eqref{eq:sufficient_decrease} should be an iteration-related one, say $\epsilon_k$. Then, we don't necessary arrive at a contradiction (to the boundedness of $F$) since the term $\sum_{k=\bar{K}+1}^\infty \epsilon_k$ could be finite.

To summarize, since the subproblems of the cyclic-BCD-based SNMF algorithm\cite{Gill2015} may have multiple solutions, there is a lack of theoretical guarantee for the convergence of the cyclic-BCD-based SNMF algorithm in \cite{Gill2015}. However, the proposed sBSUM  is guaranteed to reach the stationary solutions of the SNMF problem \cite{Razav2013}. Moreover, it is interesting to note that the proposed sBSUM  (also including the other BSUM-based SNMF algorithms to be discussed later) cannot get stuck at {\it local maxima}, as shown in Appendix B.
\subsection{Vector-wise BSUM}
Generally speaking, for multi-block BCD/BSUM algorithms, the less the number of block variables is, the faster convergence the BCD/BSUM algorithms achieve. In the following, we develop a new algorithm, termed vector-BSUM (vBSUM), in which each time we update {\it a single row} of $\bX$ by minimizing an upper bound of $F(\bX)$.

We first study the subproblem of the vBSUM  which optimizes the $i$-th row of $\bX$ (denoted as $\bx^T$ for convenience) given the current $\bX$ denoted as $\tX$. To do so, let us define $\overline{\bX}_i \triangleq\tX_{1:i-1,:}$, $\underline{\bX}_i \triangleq\tX_{i:n,:}$, $\overline{\bm{m}}_i \triangleq \bM_{1:i-1,i}$, and $\underline{\bm{m}}_i \triangleq\bM_{i+1:n,i}$. Since we have
\begin{equation}
\left(\begin{array}{c}
\overline{\bX}_i\\
\bx^T\\
\underline{\bX}_i
\end{array}\right)\left(\begin{array}{c}
\overline{\bX}_i\\
\bx^T\\
\underline{\bX}_i
\end{array}\right)^T=\left(\begin{array}{ccc}
\overline{\bX}_i\overline{\bX}_i^T & \overline{\bX}_i\bx & \overline{\bX}_i\underline{\bX}_i^T\\
\bx^T\overline{\bX}_i^T& \bx^T\bx& \bx^T\underline{\bX}_i^T\\
\underline{\bX}_i\overline{\bX}_i^T & \underline{\bX}_i\bx &\underline{\bX}_i\underline{\bX}_i^T
\end{array}\right),
\end{equation}
the subproblem for updating the $i$-th row of $\bX$  is given by
\begin{equation}\label{eq:SNMF-rBSUM}
\begin{split}
\bX_{i:}^T=&\arg\min_{\bx}2\left(\Vert\overline{\bm{m}}_i-\overline{\bX}_i\bx\Vert^2+\Vert\underline{\bm{m}}_i-\underline{\bX}_i\bx\Vert^2\right)\\
&~~~~~~~~~~~~~~~~+(M_{ii}-\norm{\bx}^2)^2\\
&\st~ \bx \geq 0.
\end{split}
\end{equation}
Define $\bP_i\triangleq\underline{\bX}_i^T\underline{\bX}_i+\overline{\bX}_i^T\overline{\bX}_i$, $\bQ_i\triangleq\bP_i-M_{ii}\bI$, and $\bq_i\triangleq(\overline{\bX}_i^T\overline{\bm{m}}_i+\underline{\bX}_i^T\underline{\bm{m}}_i)$. Then we can rewrite problem \eqref{eq:SNMF-rBSUM} equivalently (equivalent in the sense that they have the same optimal solution) as follows
\begin{equation}\label{eq:SNMF-rBSUM2}
\begin{split}
&\bX_{i:}^T=\arg\min_{\bx}\norm{\bx}^4+2\bx^T\bQ_i\bx-4\bq_i^T\bx\\
&\st~ \bx \geq 0.
\end{split}
\end{equation}
Clearly, problem \eqref{eq:SNMF-rBSUM2} is not necessarily convex. But even if it is convex\footnote{When the matrix $\bM$ is a similarity matrix generated by the widely used Gaussian kernel, e.g., $M_{ij}=\exp\left(-\frac{\Vert\bx_i-\bx_j\Vert^2}{\sigma^2}\right)$ (where $\bx_i$ and $\bx_j$ are two data points, and $\sigma^2$ is a parameter), or a \emph{normalized} similarity matrix, we always have $M_{ii}\leq 1$ for all $i$. On the other hand, because $\bX$ is nonnegative, the matrix $\bP_i$ is almost always a positive matrix whose eigenvalues are all in the interval given by $[\min(\bP_i\bone)~~\max(\bP_i\bone)]$\cite[Chap. 8]{Mtx_analysis}. Moreover, when the number of data points $n$ is very large, it is very likely that $\min(\bP_i\bone)>1$ and thus the matrix $\bQ_i=\bP_i-M_{ii}\bI$ is positive semidefinite, implying that problem \eqref{eq:SNMF-rBSUM2} is convex with very high probability in practice.}, it is still unlikely to obtain a closed-form optimum solution to problem \eqref{eq:SNMF-rBSUM2} due to the presence of the nonnegativity constraints and the coupling of the entries of $\bx$ introduced by $\bQ_i$. %and the fourth-order term $\norm{\bx}^4$.
In what follows, we solve \eqref{eq:SNMF-rBSUM2} by using BSUM algorithm. %The basic idea is to remove the coupling %{\color{red}[this paragraph seems confusing. suggest to delete. we are approximate this anyway, no need to mention covnexity? but should mention the non-separability of the problem, and that we are only approximating the possibly non-convex part.]}.
%of the entries of $\bx$.

Thanks to the Lipschitz continuity of the quadratic part in the objective function,  we can find an upper bound for the objective function of problem \eqref{eq:SNMF-rBSUM2}. Specifically, we have the following bound for the quadratic term $\bx^T\bQ_i\bx$
$$\bx^T\bQ_i\bx\leq \by^T\bQ_i\by+2\by^T\bQ_i(\bx-\by)+S_{\bQ_i}\norm{\bx-\by}^2, \forall \bx, \by.$$
where $S_{\bQ_i} $ is  the  Lipschitz constant of the gradient of $\bx^T\bQ_i\bx$, which can be simply chosen as the maximum eigenvalue of $\bQ_i$ or $\max(\bP_i\bone){-}M_{ii}$ (see footnote 3 for reason). By replacing $\bx^T\bQ_i\bx$ in \eqref{eq:SNMF-rBSUM2} with the above upper bound evaluated at $\tx\triangleq\tX_{i:}^T$ (i.e., let $\by=\tx$ in the above inequality) in problem \eqref{eq:SNMF-rBSUM2} and rearranging the terms, we obtain
\begin{equation}\label{eq:SNMF-rBSUM3}
\begin{split}
&\min_{\bx} \; \norm{\bx}^4+2S_{\bQ_i}\norm{\bx}^2-4\bb_i^T\bx\\
&\st~ \bx \geq 0.
\end{split}
\end{equation}
where $\bb_i\triangleq\bq_i+s_{\bQ_i}\tx-\bQ_i\tx$. Observing that $\bb_i^T\bx$ should be maximized for any fixed $\norm{\bx}$, we show in Lemma \ref{lem:unisol} that problem \eqref{eq:SNMF-rBSUM3} admits a unique closed-form solution. %which is summarized in Lemma \ref{lem:unisol}.
\begin{lemma}\label{lem:unisol}
%Let $p=s_{\bQ_i}$ and $q=-\norm{[\bb_i]_{+}}$. Then
The optimum solution to problem \eqref{eq:SNMF-rBSUM3} can be expressed as follows
\begin{equation}\label{eq:Xisol}
\bX_{i:}^T = \left\{\begin{split}
&\bzero,~~~~~~~~~~~~~\textrm{if}~\bb_i\leq 0\\
&t\frac{[\bb_i]_{+}}{\norm{[\bb_i]_{+}}}~~~~~~\textrm{otherwise}
\end{split}\right.
\end{equation}
where
\begin{align*}
&t=\sqrt[\leftroot{-1}\uproot{9}3]{\frac{\norm{[\bb_i]_{+}}}{2}-\sqrt{\Delta}}+\sqrt[\leftroot{-1}\uproot{9}3]{\frac{\norm{[\bb_i]_{+}}}{2}+\sqrt{\Delta}}\\ \mbox{with}\quad &\Delta\triangleq\frac{\norm{[\bb_i]_{+}}^2}{4}+\frac{S_{\bQ_i}^3}{27}.
\end{align*}
%is the unique solution to the cubic equation $t^3+s_{\bQ_i}t -\norm{[\bb_i]_{+}}=0$.
\end{lemma}
\begin{proof}
It is easily seen that when an entry of $\bb_i$ is non-positive, the corresponding component of the optimal $\bx$ should equal to zero. As a result, problem \eqref{eq:SNMF-rBSUM3} is equivalent to the following
\begin{equation}\label{eq:SNMF-rBSUM4}
\begin{split}
&\min_{\bx}\norm{\bx}^4+2S_{\bQ_i}\norm{\bx}^2-4[\bb_i]_{+}^T\bx, \quad \st \; \; \bx \geq 0.
\end{split}
\end{equation}
Trivially, we have the optimal $\bx=\bzero$ if $\bb_i\leq 0$. Hence, we only need to consider the case when $[\bb_i]_{+}\neq 0$. First, it is readily seen that, problem \eqref{eq:SNMF-rBSUM4} is further equivalent to
\begin{equation}\label{eq:SNMF-rBSUM5}
%\begin{split}
\min_{\bx,t}\; t^4+2S_{\bQ_i}t^2-4[\bb_i]_{+}^T\bx, \quad \st~ \norm{\bx}\leq t,\; \bx \geq 0.
%\end{split}
\end{equation}
Second, note that, (for any fixed $t>0$) $[\bb_i]_{+}^T\bx$ should be maximized subject to $\norm{\bx}\leq t$ and $\bx\geq 0$. By Cauchy-Schwartz inequality,  the optimal $\bx$ takes the form $\bx=t\frac{[\bb_i]_{+}}{\norm{[\bb_i]_{+}}}$. Hence, problem \eqref{eq:SNMF-rBSUM5} reduces to the following convex problem
\begin{equation}\label{eq:SNMF-rBSUM6}
\begin{split}
&\min_{t}\; t^4+2S_{\bQ_i}t^2-4\norm{[\bb_i]_{+}}t, \quad \st\; \; t\ge 0.
\end{split}
\end{equation}
By the first-order optimality condition, we know that the optimal $t$ is the unique real root of the cubic equation $t^3+ S_{\bQ_i}t-\norm{[\bb_i]_{+}}=0$, which can be obtained in closed-form as shown in above. This completes the proof.
\end{proof}
%Our proposed algorithm required an approximate solution of  problem \eqref{eq:SNMF-rBSUM2}. by iteratively solving \eqref{eq:SNMF-rBSUM3}  for a given number of times.
%By repeatedly solving \eqref{eq:SNMF-rBSUM3} with the above closed-form solution, we can obtain at least a stationary solution to problem \eqref{eq:SNMF-rBSUM2}. {\color{red} In practice, due to footnote 2, we can expect an optimum solution for problem \eqref{eq:SNMF-rBSUM2}.}
\begin{table}
\centering
\caption{Algorithm 2: vBSUM algorithm for SNMF}
\begin{tabular}{|p{3.5in}|}
\hline
\begin{itemize}
%\item \textbf{Input}: $\{\bh_{jk}\}_{j=1}^K$, $g_k(\bm{\mI}_k^r)$, $\{\mI_{jk}\}_{j\neq k}$
%\item \textbf{output}: $\{\nu_{jk}\}_{j=1}^K$
%\hline\\
\item [0.] initialize $\bX$ and calculate $\bX^T\bX$%,  and set $a=4$% $\tQ=\bQ$, $\tS=\bS$, and and $\bR=\bH_{RR}^H\bQ$.
%\item [1.]\; calculate $\bX$
\item [1.]\; \textbf{repeat}
\item [2.] \; \quad\quad \textbf{for} each $i\in\{1,2\ldots,n\}$
%\item [3.]\;\quad\quad\quad $\overline{\bX}_i=\bX_{1:i-1,:}$, ~~$\underline{\bX}_i =\bX_{i:n,:}$
%\item [4.]\;\quad\quad\quad  $\overline{\bm{m}}_i = \bM_{1:i-1,i}$,~~ $\underline{\bm{m}}_i =\bM_{i+1:n,i}$.
\item [3.]\;\quad\quad\quad $\bP_i = (\bX^T\bX)-\bX_{i:}^T\bX_{i:}$
%\item [4.]\; \quad\quad\quad
%    $\bQ_i = \bP_i-M_{ii}\bI$;
    %$\bQ_i=\underline{\bX}_i^T\underline{\bX}_i+\overline{\bX}_i^T\overline{\bX}_i-M_{ii}\bI$ and
    \item [4.]\; \quad\quad \quad
        $\bq_i=\bX^T\bM_{:i}-M_{ii}\bX_{i:}^T$
        %$\bq_i=(\overline{\bX}_i^T\overline{\bm{m}}_i+\underline{\bX}_i^T\underline{\bm{m}}_i)$,
\item [5.]\; \quad\quad \quad \textbf{for} $k=1:I_{\max}$~~//repeat Steps 6-7 $I_{\max}$ times
\item [6.]\; \quad\quad \quad \quad $\bb_i=\bq_i+(s_{\bQ_i}+M_{ii})\bX_{i:}^T-\bP_i\bX_{i:}^T$
\item [7.]\; \quad\quad \quad\quad update $\bX_{i:}$ according to \eqref{eq:Xisol} %$\bX_{i:} = t\frac{[\bb_i]_{+}}{\norm{[\bb_i]_{+}}}$
    \item[8.]\; \quad\quad \quad \textbf{end}
\item[9.]\; \quad\quad \quad $(\bX^T\bX)=\bP_i+\bX_{i:}^T\bX_{i:}$
%\item[9.]\; \quad\quad \quad $(\bX^T\bM)_{:i}=\bq_i+M_{ii}\bX_{i:}^T$
\item[10.]\;  \quad\quad \textbf{end}
\item [11.]\; \textbf{until} some termination criterion is met
%\item [7.]\; each transmitter sends $\rho_k$ to its receiver
\end{itemize}
\\
\hline
\end{tabular}%\vspace{-7pt}
\end{table}

Our proposed vBSUM algorithm, summarized in Table II, requires an approximate solution of problem \eqref{eq:SNMF-rBSUM2} at each iteration. Such solution is obtained by iteratively solving \eqref{eq:SNMF-rBSUM3}  for a fixed number of times $I_{\max}$; see Steps 5-8. %Steps 5-8 are carried out to solve problem \eqref{eq:SNMF-rBSUM2}. In practice, Steps 5- 8 are only carried out for a fixed $I_{\max}$ times to speed up the problem \eqref{eq:SNMF-rBSUM2} is not solved exactly. That is, Steps 6-7 is only performed a finite number of times, denoted as $I_{\max}$.  It is worth mentioning that, when we set $I_{\max}=1$, the vBSUM  is just the conventional BSUM algorithm. Moreover, although the vBSUM  with $I_{\max}>1$ is a generalization of the conventional BSUM algorithm,
We comment that regardless of the number of inner iterations performed (i.e., the choice of $I_{\max}$),  vBSUM is guaranteed to converge to stationary solutions of the basic SNMF problem, by applying the same analysis as that of the BSUM algorithm \cite[Theorem 2]{Razav2013}. %regardless of the number of inner iterations performed (i.e., the value of $I_{\max}$)%, the vBSUM  in Table II is guaranteed to have convergence to stationary solutions of the basic SNMF problem.
Furthermore, from Table II, one can see that the most costly step in the vBSUM  lies in Step 4 for computing $\bX^T\bM_{:i}$, equivalently $\bX^T\bM$ in each iteration, which requires $O(rn^2)$ operations per-iteration in the dense case while $O(rK)$ operations per-iteration in the sparse case. Hence, the per-iteration computational complexity of the vBSUM  is the same as the sBSUM. In addition, it is seen that we only need to store $(\bX^T\bX)$, $\bX$, $\bM$, and $(\bq_i, \bb_i)$ in the algorithm, which require $O(r^2)$, $O(nr)$, $O(K)$ and $O(r)$ space in memory, respectively. Hence, the vBSUM requires less space in memory in practice than that required by the sBSUM algorithm, though both with the same order of memory complexity, i.e., $O(\max(K, nr))$.

\begin{remark}
Certainly, we can use the element-wise BSUM algorithm (i.e., view each row entry of $\bX$ as one block) to update each row of $\bX$ and obtain an alternative vBSUM algorithm, which is a simple variant of the sBSUM algorithm obtained by (for each $i$) updating $X_{ij}$, $j=1,2,\ldots,r$, multiple times in each iteration. For convenience, we refer to this alternative vBSUM algorithm as v-sBSUM algorithm and also assume $I_{\max}$ repeats as in the vBSUM algorithm for updating rows of $\bX$. Note that, the v-sBSUM  algorithm requires $O(n^2rI_{\max})$ operations in each iteration due to the frequent computation of $\bM\bX$ (cf. Step 5 of Algorithm 1) while $O(nr^2+nr^2I_{\max})$ operations in the vBSUM algorithm, where the second term $O(nr^2I_{\max})$ is due to Step 6 of Algorithm 2. Moreover, the v-sBSUM algorithm requires $r$ times root operations as many as the vBSUM algorithm does. Therefore, we prefer the proposed vBSUM algorithm over the simple variant of the sBSUM algorithm for better efficiency.
%
%. First, similarly, let us express the $i$-th row of $\bX$ obtained after updating the $(i,j)$-th entry as $\bx=\tX_{i:}^T-\tX_{ij}\be_j+x\be_j$. With this expression, we can write the objective function of problem \eqref{eq:SNMF-rBSUM2} equivalently as
%% $h(\bx) = \norm{\bx}^4+2\bx^T\bQ_i\bx-4\bq_i^T\bx$.
%\begin{equation}
%\begin{split}
%h(x)&\triangleq \norm{\bx}^4+2\bx^T\bQ_i\bx-4\bq_i^T\bx\\
%&=x^4+a_2x^2+a_1x+a_0
%\end{split}
%\end{equation}
%where $a_2=2\left(\norm{\tX_{i:}^T-\tX_{ij}\be_j}^2+(\bQ_i)_{jj}\right)$, $a_1=4\tX_{i:}(\bQ_i)_{:j}-4(\bq_i)_j$, and $a_0$ is a constant independent of $x$. %$$a_0 = \norm{\tX_{i:}^T-\tX_{ij}\be_j}^4+2(\tX_{i:}^T-\tX_{ij}\be_j)^T\bQ_i(\tX_{i:}^T-\tX_{ij}\be_j)-4\bq_i^T(\tX_{i:}^T-\tX_{ij}\be_j).$$
\end{remark}
\begin{remark}
The vBSUM algorithm is a row-wise BSUM algorithm. Similarly, we can develop a column-wise BSUM algorithm for the basic SNMF problem. However, unlike the subproblem \eqref{eq:SNMF-rBSUM2} in the vBSUM algorithm, the subproblem of updating each column of $\bX$  has no good structure and the corresponding Lipschitz constant is not easily available as the number of data points $n$ is very large. And more importantly, the vBSUM algorithm is more amendable to both randomized and parallel implementation, which will be clear in the following sections.
\end{remark}

%\begin{lemma}
%The unique real root to the cubic equation $t^3+pt+q=0$ with $p>0$ is given by
%$$t = \sqrt[\leftroot{-1}\uproot{9}3]{-\frac{q}{2}-\sqrt{\Delta}}+\sqrt[\leftroot{-1}\uproot{9}3]{-\frac{q}{2}+\sqrt{\Delta}}$$
%where $\Delta\triangleq\frac{q^2}{4}+\frac{p^3}{27}$
%\end{lemma}

\section{Permutation-based Randomized BSUM Algorithm And Its Application in SNMF}
The sBSUM and vBSUM algorithms both fall into the category of deterministic cyclic BSUM algorithms. As a variant of cyclic BSUM algorithm, randomized BCD/BSUM algorithm has been proposed\cite{RazavPHD}, where each time one block variable is chosen to be optimized with certain probability. Differently from the basic randomized BCD/BSUM algorithm, the permutation-based randomized BCD/BSUM algorithm (termed as PR-BCD/BSUM) updates the block variables in a random permutation rule, in which the blocks are randomly selected without replacement. However, the convergence of the PR-BCD/BSUM algorithm has not been well-understood, as pointed out in a recent survey \cite{Wright2015}. In particular, it is not known, at least in theory, whether random permutation provides any added benefit compared with the more classical cyclic block selection rules. In this section, we first study the convergence of the PR-BSUM algorithm (including PR-BCD as a special case) in a general framework and then propose the randomized sBSUM and vBSUM algorithms.
\subsection{The PR-BSUM algorithm and the convergence results}
We start with a general description of the PR-BSUM algorithm. Consider the following multi-block minimization problem
\begin{equation}\label{eq:mb-minP}
\min_{\bx_i\in \cX_i, \forall i} f(\bx_1, \bx_2, \ldots, \bx_m)
\end{equation}
where each $\cX_i\in \Rdom^{n_i}$ is a closed convex set. Define $\cX \triangleq \cX_1\times \cX_2\times\ldots\times \cX_m$ and $\bx=[\bx_1^T~\bx_2^T~\ldots~\bx_m^T]^T$. Let $u_i(\cdot; \bx)$ satisfy the following assumption.
\begin{assump}\label{assump1}
\begin{subequations}
\begin{align}
&u_i(\bx_i; \bx)=f(\bx),\forall \bx\in \cX, \forall i;\label{eq:fun_assump}\\
&u_i(\by_i; \bx)\geq f(\bx_{<i}, \by_i, \bx_{>i}), \forall \by_i\in \cX_i, \forall \bx\in \cX, \forall i;\label{eq:ub_assump}\\
&u'_i(\by_i; \bx, \bd_i)\vert_{\by_i=\bx_i}=f'(\bx; \bd), \forall \bd=(\bzero; \ldots \bd_i; \bzero; \ldots\bzero)\nonumber\\
&~~~~~~~~~~~~~~~~~~~~~~~~~~~~~~~~~~~~\st ~\bx_i+\bd_i\in \cX_i, \forall i;\label{eq:grad_assump}\\
&u_i(\by_i; \bx) ~\textrm{is continuous in}~ (\by_i, \bx), \forall i,
\end{align}
\end{subequations}
\end{assump}
where $\bx_i$ is the $i$-th block component of $\bx$ (similarly for $\by_i$ and $\by$), $\bx_{<i}$ and $\bx_{>i}$ represent the block components of $\bx$ with their indices less than $i$ or larger than $i$, respectively, $u'_i(\by_i; \bx, \bd_i)$ denotes the directional derivative of $u_i(\cdot; \bx)$ with respect to $\by_i$ along the direction $\bd_i$, and $f'(\bx; \bd)$ denotes the directional derivative of $f(\cdot)$ with respect to $\bx$ along the direction $\bd$. The assumption \eqref{eq:grad_assump} guarantees that the first order behavior of $u_i(¡¤,\bx)$ is the same as $f(\cdot)$ locally\cite{Razav2013}, hence it is referred to as the  gradient consistency assumption.

The PR-BSUM algorithm is described in Table III, where the `\emph{randperm}' function in Step 3 generates an index set $\mI$ containing a random permutation of $\{1, \cdots, m\}$ and specifies the order of the update of block variables. The PR selection rule takes advantage of both the randomized rule and the cyclic rule: it guarantees timely update of all block variables, while avoiding being stuck with a bad update sequence. %the randomization across the iterations introduce  by a fixed order, while within each iteration with the cyclic rule, the random  the randomize across the iterations provides The RP rule has a key difference compared to the  randomized rule \cite{RazavPHD}: it ensures timely update of all block variables while the randomized rule can leave some block variables unchanged for a long time. %the basic randomized selection rule---randomly selecting one block each time---employed by the basic randomized BSUM algorithm may leave some block variable not updated in a long but finite time. It has been recognized that such a minor variant of the basic randomized BSUM algorithm could sometimes perform particularly well as compared with the deterministic cyclic BSUM algorithm or the basic randomized BSUM algorithm\cite{Wright2015}. However, the convergence of the PR-BSUM algorithm remains elusive.
In the following, we prove that, with probability one (w.p.1.) the sequence generated by the PR-BSUM algorithm converges to the set of stationary solutions of problem  \eqref{eq:mb-minP}.
\begin{table}
\centering
\caption{Algorithm 3: PR-BSUM Algorithm}
\begin{tabular}{|p{3.5in}|}
\hline
\begin{itemize}
%\item \textbf{Input}: $\{\bh_{jk}\}_{j=1}^K$, $g_k(\bm{\mI}_k^r)$, $\{\mI_{jk}\}_{j\neq k}$
%\item \textbf{output}: $\{\nu_{jk}\}_{j=1}^K$
%\hline\\
\item [0.] initialize $\bx^0\in \cX$ and set $k=0$%,  and set $a=4$% $\tQ=\bQ$, $\tS=\bS$, and and $\bR=\bH_{RR}^H\bQ$.
%\item [1.]\; $k=0$
\item [1.]\; \textbf{repeat}
\item [2.] \; \quad\quad $\by=\bx^k$
\item [3.] \; \quad\quad $\mI=\mbox{randperm}(m)$
\item [4.] \; \quad\quad \textbf{for} each $i\in\mI$
\item [5.]\; \quad\quad \quad
        $\cX_i^k=\arg\min_{\bx_i\in\cX_i} u_i(\bx_i, \by)$
\item [6.]\; \quad\quad \quad set $\by_i$ to be an arbitrary element in $\cX_i^k$
\item[7.]\; \quad\quad \textbf{end}
\item[8.]\; \quad\quad $\bx^{k+1} = \by$
\item[9.]\; \quad\quad $k = k+1$
\item [10.]\; \textbf{until} some termination criterion is met
%\item [7.]\; each transmitter sends $\rho_k$ to its receiver
\end{itemize}
\\
\hline
\end{tabular}%\vspace{-7pt}
\end{table}

\begin{theorem}\label{thm:PR-BSUM}
Let Assumption \ref{assump1} holds. Furthermore, assume that $f(\cdot)$ is regular and  bounded below. Then every limit point of the iterates generated by the PR-BSUM algorithm is a stationary point of problem \eqref{eq:mb-minP} w.p.1.
\end{theorem}
\begin{proof}
There are $M=m!$ different permutations. Let $p$ denote the index of permutation and $p(1)$ denote the first number of the $p$-th permutation. First of all, we have
\begin{align}\label{eq:expect_value}
\expect[f(\bx^{k+1})\mid \bx^k]=\frac{1}{M} \sum_{p=1}^{M} f(\bx^{p,k+1})
\end{align}
where $\bx^{p,k+1}$ denotes the update obtained by running one iteration of PR-BSUM (given $\bx^k$) according to the variable selection rule specified by the $p$-th permutation. Due to the upper bound assumption \eqref{eq:ub_assump} and the update rule, it must hold that
\begin{align}\label{eq:bound_r0}
f(\bx^{p,k+1})\leq \min_{\bx_{p(1)}\in \cX_{p(1)}}u_{p(1)}(\bx_{p(1)}; \bx^k),  \; \forall  p.
\end{align}
Combining \eqref{eq:expect_value} and \eqref{eq:bound_r0}, we have
\begin{align}\label{eq:bound_r}
&\expect[f(\bx^{k+1}) \mid \bx^k]\leq f(\bx^k)-\frac{1}{M} \sum_{p=1}^{M}\left(f(\bx^k)\right.\nonumber\\
&~~~~~~~~~~~~~~\left.-\min_{\bx_{p(1)\in \cX_{p(1)}}}u_{p(1)}(\bx_{p(1)}; \bx^k)\right)
\end{align}
which implies that $f(\bx^k)$ is a supermartingale and thus converges\cite{Bertsekas1996}, and moreover the following holds w.p.1., %\footnote{The acronym a. s. is the abbreviation of almost surely.}
\begin{equation}
\frac{1}{M}\!\sum_{k=1}^\infty \sum_{p=1}^{M}\left(f(\bx^k){-}\min_{\bx_{p(1)}\in \cX_{p(1)}}u_{p(1)}(\bx_{p(1)}; \bx^k)\right)<\infty.
\end{equation}
Thus, by noting $f(\bx^k){\geq} \min_{\bx_{p(1){\in} \cX_{p(1)}}}u_{p(1)}(\bx_{p(1)}; \bx^k), \forall p$, we must have, w.p.1.,
\begin{equation}\label{eq:almost}
\lim_{k\rightarrow\infty} \left(f(\bx^k)-\min_{\bx_{p(1)\in \cX_{p(1)}}}u_{p(1)}(\bx_{p(1)}; \bx^k)\right)=0, \forall p.
\end{equation}
Now let us restrict our analysis to a convergent subsequence $\{\bx^{k_j}\}$ with $\lim_{j\rightarrow\infty}\bx^{k_j}=\bx^{\infty}$. We have from \eqref{eq:almost} and the continuity of $f(\cdot)$ that  \begin{equation}\label{eq:almost1}
\lim_{j\rightarrow\infty} \min_{\bx_{p(1)}\in \cX_{p(1)}}u_{p(1)}(\bx_{p(1)}; \bx^{k_j})=f(\bx^{\infty}), \forall p,~\textrm{\textrm{w.p.1}}.
\end{equation}
On the other hand, according to the update rule, we have
\begin{align}\label{eq:mineq}
&\min_{\bx_{p(1)}\in \cX_{p(1)}}u_{p(1)}(\bx_{p(1)}; \bx^{k_j})\leq u_{p(1)}(\bx_{p(1)}; \bx^{k_j}),  \nonumber\\
&~~~~~~~~~~~~~~~~~~~~~~~~~~~~~\forall \bx_{p(1)}\in \cX_{p(1)},\forall p,~\textrm{\textrm{w.p.1}}.
\end{align}
By taking limit as $j\rightarrow\infty$ on both sides of \eqref{eq:mineq}, and using \eqref{eq:almost1} and the continuity of $u_{i}(\cdot;\cdot)$, we obtain
\begin{equation}\label{eq:mineq1}
f(\bx^{\infty})\leq u_{p(1)}(\bx_{p(1)}; \bx^{\infty}), \forall \bx_{p(1)}\in \cX_{p(1)}, \forall p,~\textrm{\textrm{w.p.1}}.
\end{equation}
Due to the function value consistency assumption \eqref{eq:fun_assump}, we have $f(\bx^{\infty})=u(\bx_{i}^{\infty}; \bx^{\infty}), \forall i$, and thus
\begin{equation}\label{eq:mineq1}
u(\bx_{p(1)}^{\infty}; \bx^{\infty})\leq u_{p(1)}(\bx_{p(1)}; \bx^{\infty}), \forall \bx_{p(1)}\in \cX_{p(1)}, \forall p,~\textrm{\textrm{w.p.1}}.
\end{equation}
Note that the above inequality holds for all permutations. Therefore, we have that w.p.1.,
\begin{equation}\label{eq:mineq1}
u_{i}(\bx_{i}^{\infty}; \bx^{\infty})\leq u_{i}(\bx_{i}; \bx^{\infty}), \forall \bx_{i}\in \cX_{i}, \forall i.
\end{equation}
Checking the first order optimality condition combined with the gradient consistency
assumption \eqref{eq:grad_assump}, we complete the proof.
\end{proof}
\begin{remark}
It is important to note that the PR-BSUM includes the PR-BCD as a special case. Therefore, the fact that the PR-BSUM does not require the uniqueness of the per-block solution of problem  \eqref{eq:mb-minP} implies that the same holds true for the PR-BCD. %Therefore it is interesting to see that by simply changing the update rule, one can make BCD algorithm convergent without requiring unique per-block optimizers. %It is seen from the proof of Theorem \ref{thm:PR-BSUM} that, as the basic randomized BSUM algorithm, the PR-BSUM algorithm does not require the uniqueness of the minimizer of problem  \eqref{eq:mb-minP} for convergence.
It follows that a simple strategy to ensure convergence of the cyclic BCD algorithm proposed in \cite{Gill2015} is to replace the cyclic rule with the random permutation rule. %applying the RP  selection rule, the randomized BCD algorithm applied to the SNMF problem is guaranteed to have convergence to stationary solutions.
In our subsequent numerical result, we will also show that the PR-BCD/BSUM algorithms are often faster than its determinsitic cyclic counterpart. %, the permutation-based randomized BCD/BSUM algorithms perform often faster than the deterministic cyclic BCD/BSUM algorithms, which will be examined in the simulation later.
\end{remark}
\subsection{The PR-sBSUM/vBSUM algorithms}
It is easy to implement the random permutation rule in the proposed sBSUM and vBSUM algorithms.  The resulting algorithms are named respectively as PR-sBSUM and PR-vBSUM.
Note that we have $nr$ block variables in the sBSUM  and $n$ block variables in the vBSUM. Therefore, by adding `$\mbox{randperm(nr)}$' as an additional step between Step 1 and Step 2 of Algorithm 1 and meanwhile slightly modifying Step 2, we obtain the PR-sBSUM algorithm. Similarly, by adding `$\mbox{randperm}(n)$' step between Step 1 and Step 2 of Algorithm 2 and meanwhile slightly modifying Step 2, we obtain the PR-vBSUM algorithm. Since the complexity of random permutation of $N$ integers is $O(N)$\cite{randperm}, the PR-sBSUM algorithm and the PR-vBSUM algorithm have the same order of per-iteration complexity as their deterministic versions, though the permutation steps thereof incur additional computational burden.
\section{Parallel BSUM Algorithm For SNMF}
In this section we present parallel versions of sBSUM and vBSUM based upon the recent work \cite{Razav2014}. Such parallel implementation is capable of utilizing multi-core processors and can deal with problems of relatively large size.
%The above algorithms are all based on sequential update of block variables.
%The parallel version is , when the data size is huge, the serial algorithms may need a very long time to converge. To achieve faster convergence and also make use of the power of multi-core processors, in this section we develop parallel implementation of the
%sBSUM algorithm and vBSUM algorithm
To enable parallelization, the main idea is to first use parallel updates to find a good  direction,  followed by some stepsize control for updating the variables \cite{Razav2014}. %applied to problem \eqref{eq:mb-minP}.
Specifically, the parallel BSUM for solving problem \eqref{eq:mb-minP} works as follows. At the $k$-th iteration, a subset $\mathcal{S}^k$ of block variables are selected and updated in parallel according to the following rules:
\begin{equation}
\begin{split}\label{eq:parallel}
&\hat{\bx}_i^{k-1}= \arg\min_{\bx_i\in \cX_i} u_i(\bx_i, \bx^{k-1}),~ \forall i\in \mathcal{S}^k\\
&\bx_i^k= \bx^{k-1}+\gamma^k(\hat{\bx}_i^{k-1}- \bx^{k-1}),~\forall i\in \mathcal{S}^k,
\end{split}
\end{equation}
where $\gamma^k$ is a stepsize. Theoretically, with appropriate choice of stepsizes, the iterates generated by the parallel BSUM algorithm converge to the set of stationary solutions of problem \eqref{eq:mb-minP}\cite{Razav2014}. %Hence, following the basic idea of parallel BSUM algorithm, we develop parallel sBSUM algorithm and vBSUM algorithm for better efficiency.
Practically, because the computation for each block $i\in \mathcal{S}^k$ is independent of the rest, one can implement it on a single core/processor, making the overall algorithm well-suited for parallel implementation on multi-core machines or over a cluster of computing nodes.

Suppose that we have $P$ processors (each with a single core) that can be used in parallel computing. In what follows we show step-by-step how to distribute the storage and computation of variables and/or data to $P$ processors.

\noindent {\bf Preprocessing step.} Observing from Table I and II that, besides the current $\bX$, all  we need to update the $(i,j)$-th entry of $\bX$ is the $i$-th column (or row) of the matrix $\bM$.
Based on this observation, we assign the variables and data to processors as follows: we divide the rows of $\bX$ into $P$ subsets $\bX_i$, $i=1,2,\ldots,P$.   Accordingly, we also divide the rows of $\bM$ into $P$ subsets $\bM_i$, $i=1,2,\ldots,P$. Moreover, each processor stores a copy of $\bX$  and $\bM_i$ in its local memory which will be used to update the $i$-th row's entries of $\bX$. See Fig. 1 for illustration.

\noindent{\bf Computation step.} At each iteration of the algorithm each processor $i$ randomly selects a subset of local variables from $\bX_i$ to carry out the computation \eqref{eq:parallel}. In particular, a subset of {\it entries} of $\bX_i$ will be selected for the parallel sBSUM while  a subset of {\it rows}  of $\bX_i$ will be picked for parallel vBSUM.

\noindent{\bf Communication step.}  After the updates are done, the processors exchange the updated variables among themselves to form a new $\bX$ at each local memory.  %that %each processor can obtain all updated variables to maintain
%a copy of the most updated $\bX$ is maintained at each local memory. %Specifically, in each iteration of the parallel sBSUM algorithm, each processor $i$ randomly select some entries of $\bX_i$ to be updated, while in each iteration of the parallel vBSUM algorithm each processor randomly selects a subset of rows of $\bX_i$ to be updated.
To avoid communicating the entire $\bX$ among the processors,  we propose to exchange {\it both} the newly updated entries and their associated indices $(i,j)$ for the sBSUM, while only exchange the newly updated entries and the {\it row indices} for the vBSUM. It can be shown that, if we update $Jr$ entries of $\bX$ in total at each iteration, the per-iteration communication complexity of the parallel sBSUM and vBSUM  are $O(3Jr(P-1))$ and $O(J(r+1)(P-1))$, respectively. Hence, the parallel vBSUM algorithm incurs less communication overhead than the parallel sBSUM algorithm. %Certainly, the parallel sBSUM algorithm can also randomly select rows of entries of $\bX$ to be updated and thus reduces the communication amount for exchanging the indices. In this case, the two parallel methods have the same communication complexity.

\begin{figure}[t]
\centering
\includegraphics[width=3.5in]{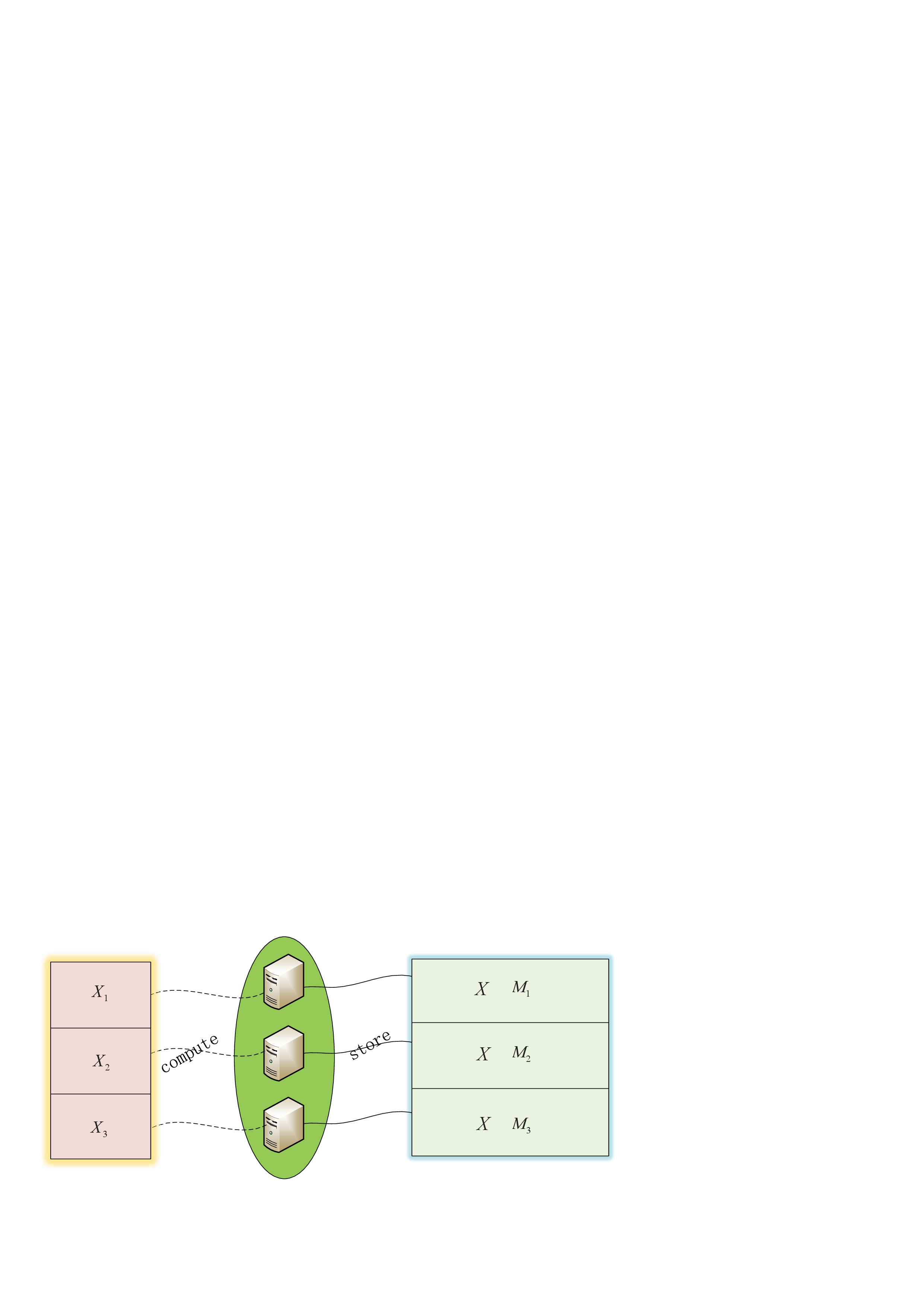}
\caption{This example shows how we assign variables/data to processors. In this example, we have three single-core processors. We divide the rows of $\bX$ into three blocks $(\bX_1,\bX_2,\bX_3)$ from top to down and assign them to the three processors for computing (the correspondence between the variables/data and processors are indicated by dotted/real lines). The size of three blocks depends on the processing capability of the corresponding processors. Furthermore, the rows of $\bM$ are accordingly divided into three blocks $(\bM_1,\bM_2,\bM_3)$ with the same size as the three blocks of $\bX$. All three processors store $\bX$ but each with only a portion of $\bM$. Hence, the processors need to exchange the updated variables to each other to locally maintain a copy of the current $\bX$.}
\label{fig:fig0}
\end{figure}

\section{Simulation Results}
This section presents numerical examples to show the effectiveness of the proposed algorithms. We first describe the simulation setup, and then demonstrate the convergence performance of various proposed algorithms. Finally, we compare the proposed algorithms with several recent state-of-the-art algorithms \cite{Kuang2012,Huang2014}.

\subsection{The Setup}
In our simulations, all algorithms are implemented in Matlab on a laptop of $8$ GB Memory and $2.10$ GHz CPU, except for the parallel BSUM algorithms. To construct the data matrix $\bM$ we utilize two different approaches suggested in  \cite{Huang2014} and \cite{Kuang2012}, respectively:

\noindent{\bf Approach 1}--{The correlation kernel (CK) method\cite{Huang2014}}: the data matrix $\bX_{data}\in \Rdom^{n\times  m}$ could be real data  or randomly generated. In the latter case, it has a fraction $s$ of zeros, and the nonzero entries all follow an i.i.d exponential distribution with unit mean. Given $\bX_{data}$, we set $\bM=\bX_{data}\bX_{data}^T+\frac{\sigma}{2}(\bN+\bN^T)$, where $\bN\in \Rdom^{n\times n}$ represents some noise matrix whose entries are randomly drawn from an i.i.d Gaussian distribution with zero mean and standard deviation $\sigma=0.1$.

\noindent{\bf Approach 2}--{The sparse Gaussian kernel (SGK) method\cite{Kuang2012}}: $\bX_{data}$ could be real data or randomly generated. In the latter case, each entry follows an i.i.d exponential distribution with unit mean. Given $\bX_{data}$, we construct $\bM$ following three steps, including 1) computing Gaussian kernel with self-tuning method, 2) sparsification, and 3) normalization. We refer readers to Sec. 7.1 of \cite{Kuang2012} for the details.

The reason to consider two different methods for generating $\bM$ is that, in our experiments, we have observed that the performance of the state-of-the-art algorithms \cite{Kuang2012,Huang2014} is significantly impacted by the way that such $\bM$ is generated. Therefore, to thoroughly investigate the performance of various algorithms, we use both methods to generate $\bM$ in our experiments.  % the works \cite{Kuang2012,Huang2014} and adopt the following two methods to generate the matrix $\bM$.

In our simulation, unless otherwise specified, the data matrix $\bX_{data}$ is randomly generated with $m=r$. If there is a value ``$s$" shown in the caption of a figure, it indicates that the CK method is used. Although the CK method generates a sparse matrix $\bX_{data}$ when $s$ is large, the corresponding $\bM$ may not be sparse due to the inner product operation and the existence of noise. In contrast, the SGK method produces sparse matrix $\bM$ due to the sparsification step. As will be seen later, the sparsity would impact the convergence rate of all the algorithms.

All algorithms are randomly initialized from points in the form of $\sqrt{\alpha}\bX_0$, where $\bX_0$ is a randomly generated nonnegative matrix and $\alpha= \arg\min_{\alpha\geq 0}\Vert\bM-\alpha\bX_0\bX_0^T\Vert^2$ is chosen to make the initial point match the scale of $\bM$. Two criteria are used to measure the performance of the algorithms. The first one is related to the objective value, given by $100 \times \Vert\bM-\bX\bX^T\Vert/\Vert\bM\Vert$\cite{Gill2015}, while the other is used to measure the gap to stationarity, given by $\Vert\bX-[\bX-\nabla F(\bX)]_{+}\Vert_{\infty}$ \cite{Razav2014}. It can be readily shown that such gap equals to zero if and only if a  stationary solution is achieved. For convenience, we refer to such gap as the \emph{optimality gap}. %{\color{red}The optimality gap measure is important but never used in the literature. [??]} It can be shown that It is readily known that, when the optimality gap vanishes, the stationary convergence is achieved.
\subsection{The convergence performance of the BSUM algorithms}
We first examine the convergence performance of the proposed algorithms.
\subsubsection{Cyclic BCD/BSUM Vs. PR-BCD/BSUM}
In this set of simulations, taking the PR-sBSUM algorithm as an example, we compare the convergence performance of the cyclic BCD/BSUM algorithm and the permutation-based randomized BCD/BSUM algorithm. The simulation results are presented in Figs. 2 and 3, where each data point is obtained by averaging $20$ problem instances. First, we can see from the two plots that the BSUM algorithms can achieve very similar (but not exactly the same) convergence performance as that of the BCD algorithms. %, for both the cyclic and the RP versions.
Second, it is observed from Fig. 2 that, when the problem is generated by the CK method with small size, the cyclic BCD/BSUM  and the PR-BCD/BSUM  could have very similar  performance. The benefit of the permutation-based random selection rule becomes significant when the problem size increases. %While the problem size is large, the performance gain of the RP-BSUM/BCD  over  the  cyclic BSUM/BCD algorithms comes, and moreover the larger the problem size, the more significant the performance gain achieved from permutation.
In contrast,  for the problems generated by the SGK method, both algorithms have similar performance even in high-dimensions; see Fig. 3. Our conclusion is that the PR-BCD/BSUM at least perform as well as its cyclic counterpart, and in certain high-dimensional scenarios using such scheme is indeed beneficial. % not always outperform its cyclic counterpart, but Therefore we conclude that   the RP-BCD/BSUM is generally better than, or at least  similar average performance as) their cyclic counterparts, especially when the problem size is very large.
%Overall, from our result %We remark that, although (as shown in Fig. 3) the RP rule is not always better than the cyclic rule even when the problem size is large, %and their performance (compared to their cyclic counterparts) could be impacted by the structure of $\bM$,
%extensive simulations indicate that the RP BCD/BSUM algorithms is generally better than (at least have similar average performance as) their cyclic counterparts, especially when the problem size is very large.
%Overall, from our result

\begin{figure*}[htbp]
\centering \subfigure[$n=20$, $r=4$]{
\label{fig:subfig:fig2a} %% label for second subfigure
\includegraphics[width=0.3\textwidth]{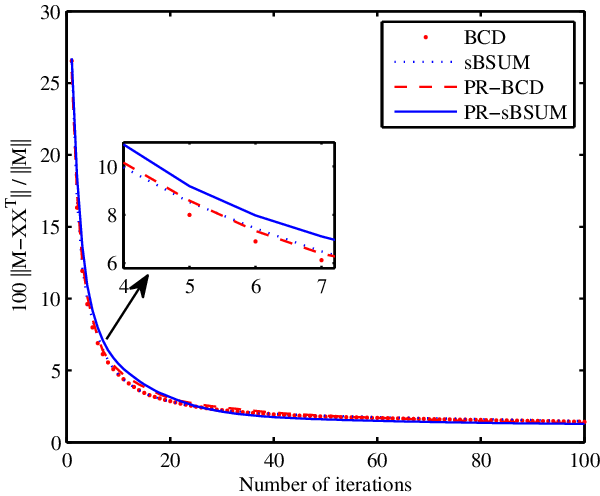}}
\subfigure[$n=100$, $r=20$]{
\label{fig:subfig:fig2b} %% label for second subfigure
\includegraphics[width=0.3\textwidth]{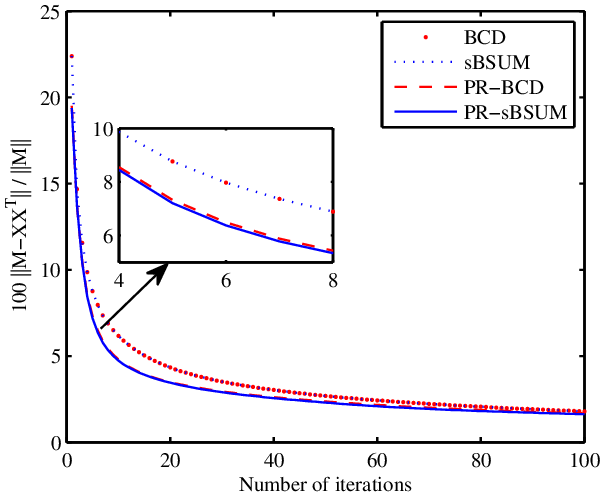}}
\subfigure[$n{=}1000$, $r{=}100$]{
\label{fig:subfig:fig2c} %% label for second subfigure
\includegraphics[width=0.3\textwidth]{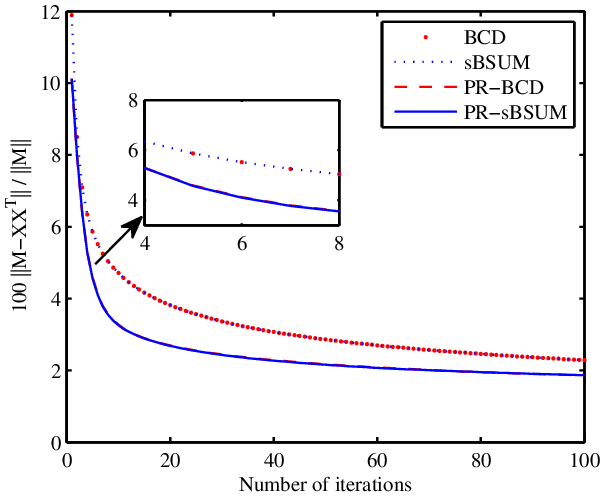}}
\label{fig:fig2}
\caption{As the problem size increases, the performance gain of the PR-BCD/BSUM  over the cyclic BCD/BSUM  becomes more significant: $s=0$.}
\end{figure*}

%\begin{figure}[htb]
%\centering
%\begin{minipage}[b]{.3\linewidth}
%  \centering
%  \centerline{\includegraphics[width=2.0in]{fig1}}
%%  \vspace{1.5cm}
%  \centerline{(a) $n=20$, $m=r=4$}\medskip
%\end{minipage}
%\hfill
%\begin{minipage}[b]{0.3\linewidth}
%  \centering
%  \centerline{\includegraphics[width=2.0in]{fig2}}
%%  \vspace{1.5cm}
%  \centerline{(b) $n=100$, $m=r=20$}\medskip
%\end{minipage}
%\hfill
%\begin{minipage}[b]{0.3\linewidth}
%  \centering
%  \centerline{\includegraphics[width=2.0in]{fig3}}
%%  \vspace{1.5cm}
%  \centerline{(b) $n=1000$, $r=100$}\medskip
%\end{minipage}
%\caption{As the problem size increases, the performance gain of the randomized BCD/BSUM algorithm over the cyclic BCD/BSUM algorithm becomes more significant.}
%\label{fig:res}
%\end{figure}

\begin{figure}[t]
\centering
\includegraphics[width=3.5in]{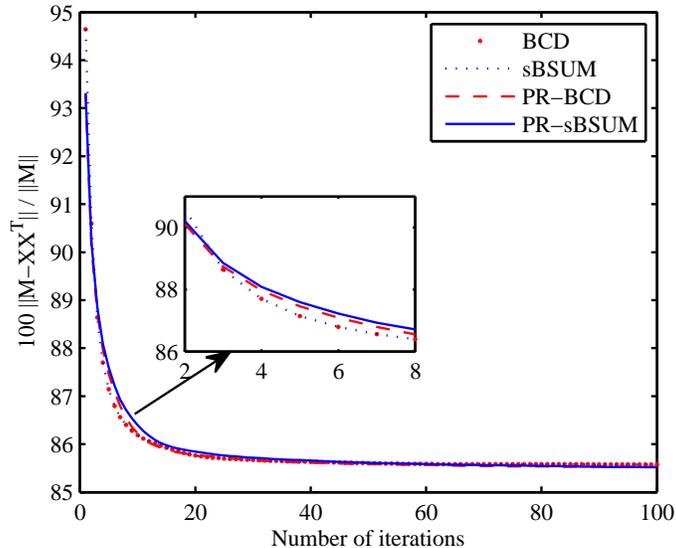}
\caption{The PR-BCD/BSUM  is not necessarily better than the cyclic BCD/BSUM  for the problems generated by the SGK method, $n=1000$, $r=100$.}
\label{fig:fig3}
\end{figure}

\subsubsection{vBSUM Vs. sBSUM/BCD}
In this set of simulations, we compare the convergence performance of the (cyclic) sBSUM/BCD and vBSUM in terms of cpu time. For vBSUM, we set $I_{\max}=10$. As shown in Fig. 1, the sBSUM and BCD  have very similar iteration convergence behavior. However, they consume different cpu time in each iteration because of different solutions applied to the cubic equations. As mentioned in Sec. II-A.2), for solving each cubic equation, a subroutine (i.e., Algorithm 1 in \cite{Gill2015}) is required by the BCD but a closed-form solution is available in the sBSUM. Hence, the sBSUM is more efficient than the BCD in terms of the consumed cpu time, as shown in Fig. 4. Moreover, it can be observed from Fig. 4 that the vBSUM outperforms the sBSUM. %Furthermore, it is seen from Fig. 4(b) that, although the BCD algorithm has not yet been proven to be able to reach stationary solutions,  as can be expected, it does works towards stationary solutions in practice as the other two algorithms do.
\begin{figure}[htbp]
\centering \subfigure[Objective function]{
\label{fig:subfig:fig4a} %% label for second subfigure
\includegraphics[width=0.5\textwidth]{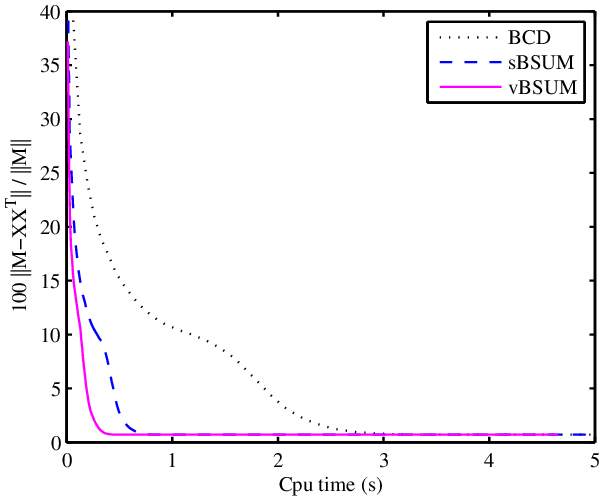}}
\hspace{-5pt} \subfigure[Optimality gap]{
\label{fig:subfig:fig4b} %% label for second subfigure
\includegraphics[width=0.5\textwidth]{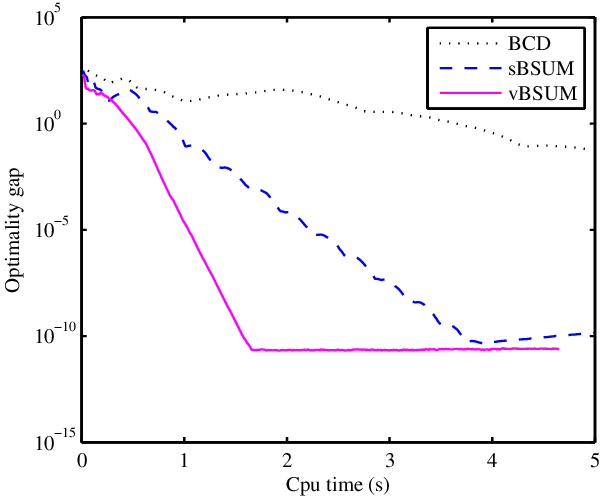}}
\label{fig:fig4}
\caption{The vBSUM  outperforms the sBSUM  while both are faster than the BCD: $n=100$, $r=10$, $s=0.5$.}
\end{figure}

\subsubsection{Parallel BSUM Vs. Serial BSUM}
In this set of simulations, we compare the performance of the parallel BSUM algorithms and the serial BSUM algorithms (i.e., the sBSUM algorithm and vBSUM algorithm). A constant stepsize $\gamma=1$ is used for all parallel BSUM algorithms. We use the noiseless CK method and the 20newsgroup text data\footnote{This dataset contains binary occurrences for $100$ words across $16242$ postings, which is available from http://www.cs.nyu.edu/$\backsim$roweis/data/20news\_w100.mat.} to generate the similarity matrix $\bM$ with $n=16242$ and $r=10$. Both the parallel BSUM  and serial BSUM  are implemented on a Condo Cluster consisting primarily of $176$ SuperMicro servers each with two $8$-core Intel Haswell processors, 128 GB of memory and 2.5 TB of available local disk. %We test the parallel algorithms with different number of cores.

The simulation results are presented in Fig. 5, where Fig. 5(a) shows convergence performance of the parallel BSUM  while Fig. 5(b) shows the cpu time consumed for running $20$ iterations of each algorithm\footnote{For a clear demonstration of the cpu time, we run the parallel algorithms $20$ iterations, which are generally enough for the algorithms to achieve a good solution.} decreasing with the number of cores used in the parallel computation. It is observed that parallelization does significantly reduces the overall computational time.  %can achieve much faster convergence than the serial BSUM algorithms.
%The more cores used, the more speedup achieved.
For example, when 256 cores are used, the vBSUM algorithm takes only less than $0.5$ second to complete $20$ iterations of computation. Furthermore, we can see that the parallel vBSUM  is faster than the parallel sBSUM. This is partly because that the vBSUM requires less communication overhead than the sBSUM, as discussed in the end of Sec. IV. In addition, it is worth mentioning that, the cpu time does not scale linearly with respect to the number of cores -- a  result predicted by the Amdahl's law \cite{Amdahl}.

%the convergence rate improves drastically as the number of cores increases from $4$ to $16$

%the more processors used, the better performance achieved, as long as the computational time dominates the communication time among processors\footnote{Intuitively, the more processor used, the faster convergence achieved. However, as the number of cores increases, the cores need to frequently exchange information to each other. As a result, the communication time would dominate the computational time when the number of cores is very large. In this case, the convergence rate will decrease when the number of cores further increases. Hence, there is a tradeoff between the convergence rate and the number of cores.}.
\begin{figure}[htbp]
\centering \subfigure[Objective function Vs. cpu time]{
\label{fig:subfig:fig5a} %% label for second subfigure
\includegraphics[width=0.5\textwidth]{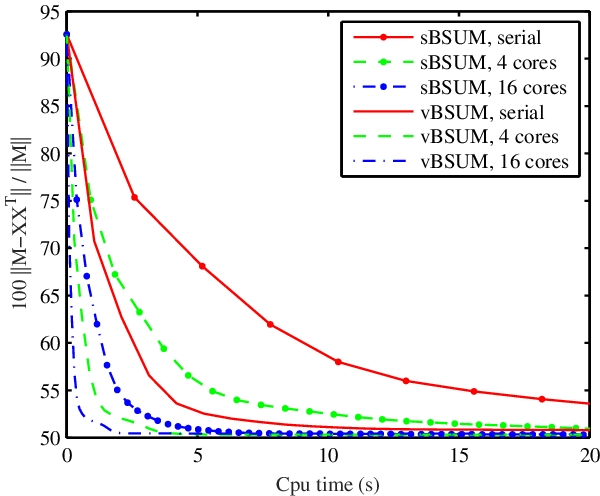}}
\vspace{15pt}
\subfigure[Cpu time Vs. number of cores]{
\label{fig:subfig:fig5b} %% label for second subfigure
\includegraphics[width=0.5\textwidth, height=0.4\textwidth]{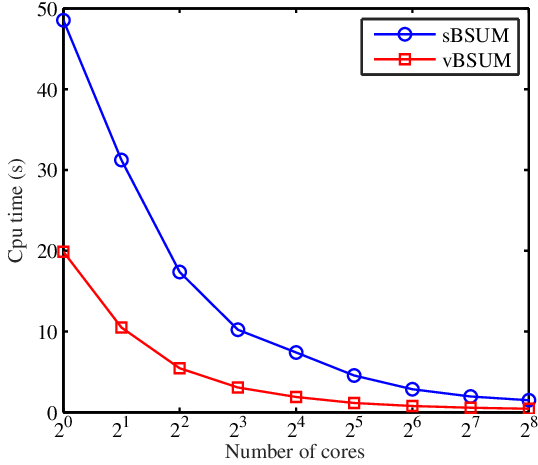}}
\label{fig:fig5}
\caption{Parallel BSUM algorithms are much faster than serial BSUM algorithms and their cpu time required for $20$ iterations decreases with the number of cores involved in the parallel computation.}
\end{figure}

\subsection{Comparison with state-of-the-art algorithms}
We compare the proposed algorithms with several state-of-the-art SNMF algorithms, listed as follows:
\begin{itemize}
\item \emph{The Newton Method}\cite{Kuang2012}: a Newton-like method with per-iteration complexity of $O(n^3r)$. It can achieve convergence to stationary solutions of the SNMF problem.
\item \emph{The ANLS Method}\cite{Kuang2012}: a penalty method that addresses the following penalty problem
 $$\min_{\bX,\bY\geq 0} \Vert\bM-\bX\bY^T\Vert^2+\alpha\Vert\bX-\bY\Vert^2,$$  which is based on a key fact that $\bX{-}\bY{\rightarrow} 0$ as the penalty parameter $\alpha$ goes to infinity. For a tuned penalty parameter $\alpha$, a two-block coordinate descent algorithm is used to solve the penalty problem, leading to $\bX$-subproblem and $\bY$-subproblem (corresponding to Eqs. (16) and (17) in \cite{Kuang2012}) in each iteration. The X(Y)-subproblem can be further decomposed into $n$ independent nonnegativity-constrained least-square (NLS) problems in the form of $\min_{\bx_i\geq 0} \Vert\bA\bx_i-\bb_i\Vert^2$ with $\bA\in \Rdom^{(n+r)\times r}$ and $\bx_i\in \Rdom^{r\times 1}$, $i=1,2,\ldots, n$. It is not difficult to see that each NLS problem have computational complexity at least in the same order as that of solving its unconstrained counterpart (i.e., computing $\bx_i = \bA^{\dag}\bb_i$), which requires at least $O((n+r)r)$ operations. Hence, the per-iteration complexity of the ANLS method is at least $O(n^2 r)$, the same order as that of the vBSUM/sBSUM algorithms.
\item \emph{The rEVD Method}\cite{Huang2014}: an approximate method that is based on \emph{reduced} eigenvalue decomposition (EVD) and unitary rotation. This method first performs rank-$r$ reduced EVD on $\bM$, yielding $\bM \approx \bU_r\bSigma_r\bU_r^T$ where $\bSigma_r$ is a $r$ by $r$ diagonal matrix of the $r$ largest eigenvalues and $\bU_r$ is an $n$ by $r$ matrix whose columns are the corresponding eigenvectors. And then an additional step is performed to find an approximate nonnegative matrix factorization $\bX$ of $\bM$ by solving
    \begin{equation}\label{eq:rotation}
    \begin{split}
    &\min_{\bX, \bQ} \Vert\bX-\bB\bQ\Vert^2\\
    &\st~ \bX\geq 0\\
    &~~~~~~\bQ^T\bQ=\bQ\bQ^T=\bI
    \end{split}
    \end{equation}
    where $\bB\triangleq\bU_r\bSigma_r^{\frac{1}{2}}$ is given and $\bQ$ is a $r$ by $r$ \emph{unitary} matrix. Problem \eqref{eq:rotation} is addressed using two-block coordinate descent algorithm with per-iteration complexity $O(nr^2)$\cite{Huang2014}. Although with lower per-iteration complexity, the rEVD method requires $O(n^3)$ operations due to EVD. Moreover, it cannot guarantee convergence to stationary solutions of the SNMF problem when the matrix $\bM$ does not have an exact SNMF. %\textcolor{red}{This method requires $\bM$ to have $r$ positive eigenvalues}
\end{itemize}

The codes of Newton and ANLS methods are available from http://math.ucla.edu/~dakuang/ (hence we use the default tuning strategy for $\alpha$ recommended by the authors of \cite{Kuang2012}), while the rEVD method is implemented according to the pseudo-code shown in Fig 4 of \cite{Huang2014}. The simulation results are presented in Figs. 6-9, from which, the following observations are made:
\begin{itemize}
\item Due to the linear per-iteration complexity with respect to $n$, the rEVD algorithm always converges fastest in all cases and often generate a good approximation solutions with reasonably small objective values. However, the rEVD algorithm cannot reach stationary solutions in all cases. Moreover, it cannot generate any solutions until the EVD operation is completed. For example, as shown in Fig. 9, the rEVD algorithm starts producing solutions in about 20 seconds, while in the mean time, the vBSUM/sBSUM algorithm has output an approximate solution with smaller objective value than the final solutions of the rEVD algorithm. This fact stands the way of adoption of the rEVD algorithm for large-scale problems. But it is worth mentioning that, the rEVD algorithm can serve as a good initialization for the BSUM/BCD algorithms when the problem size is relatively small.
\item The ANLS algorithm could sometimes perform better than the BSUM algorithms. The main reason is that the ANLS algorithm is in essence two-block BCD, which allows it to benefit from MATLAB's high-performance matrix computation. However, the ANLS algorithm may converge to a bad solution (see Fig. 6). This is because that the adaptive penalty parameter provided by the authors  of \cite{Kuang2012} may fail to work effectively.
\item The Newton method often works very well for small-scale problems, though its optimality gap value could reach a relatively high error floor. However, it could converge extremely slowly for large-scale problems due to its high per-iteration complexity (see Fig. 8). This fact stands the way of its adoption for large-scale problems.
\item In all cases, the proposed BSUM algorithms (especially the vBSUM algorithm) can work very well and quickly reach a good approximate solution in the first few seconds or iterations. Moreover, we find that, the BSUM algorithms can converge faster for the problems with sparse similarity matrices as compared to those with dense similarity matrix. For example, in Fig. 8 (corresponding to a dense similarity matrix $\bM$ generated by the CK method), the vBSUM algorithm takes about $600$ seconds to converge while in Fig. 9 (corresponding to a sparse similarity matrix generated by the SGK method) it takes only about $60$ seconds for the problem of same size.
\end{itemize}

To summarize, the proposed BSUM-type algorithms can reach stationary solutions and are more numerically stable than the state-of-the-art algorithms regardless the way the matrix $\bM$ is generated. Moreover, parallel computation can significantly improve the efficiency of the proposed methods.
%\subsection{The performance of parallel algorithms}
%Parallel sBSUM versus vBSUM
%\subsection{Clustering performance: NMF Vs. SNMF}
\begin{figure}[htbp]
\centering \subfigure[Objective function]{
\label{fig:subfig:fig4a} %% label for second subfigure
\includegraphics[width=0.5\textwidth]{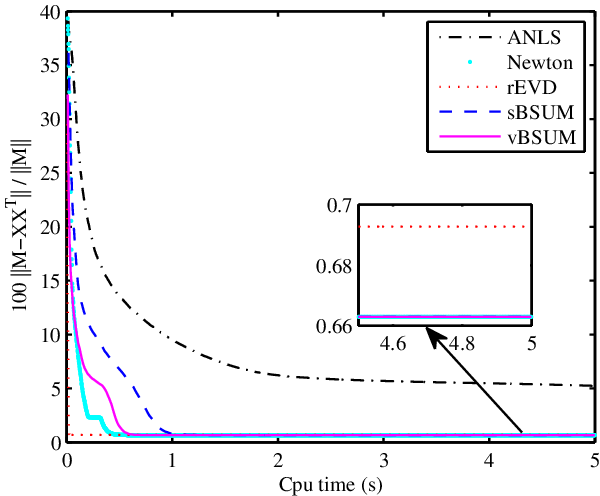}}
\hspace{-5pt} \subfigure[Optimality gap]{
\label{fig:subfig:fig4b} %% label for second subfigure
\includegraphics[width=0.5\textwidth]{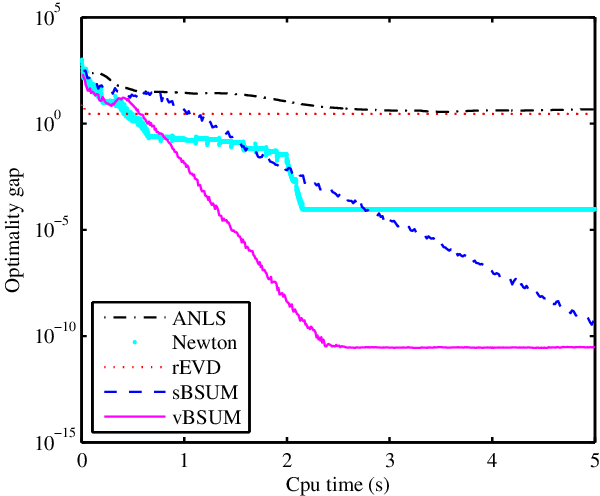}}
\label{fig:fig6}
\caption{The ANLS and rEVD algorithms may fail to converge to stationary solutions: $n=100$, $r=10$, $s=0.5$.}
\end{figure}

\begin{figure}[htbp]
\centering \subfigure[Objective function]{
\label{fig:subfig:fig4a} %% label for second subfigure
\includegraphics[width=0.5\textwidth]{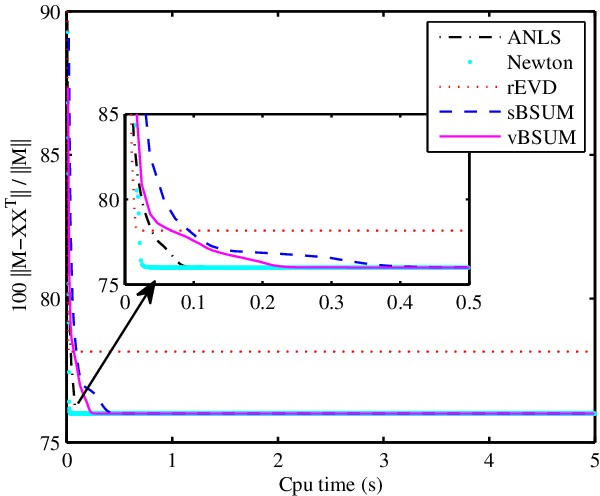}}
\hspace{-5pt} \subfigure[Optimality gap]{
\label{fig:subfig:fig4b} %% label for second subfigure
\includegraphics[width=0.5\textwidth]{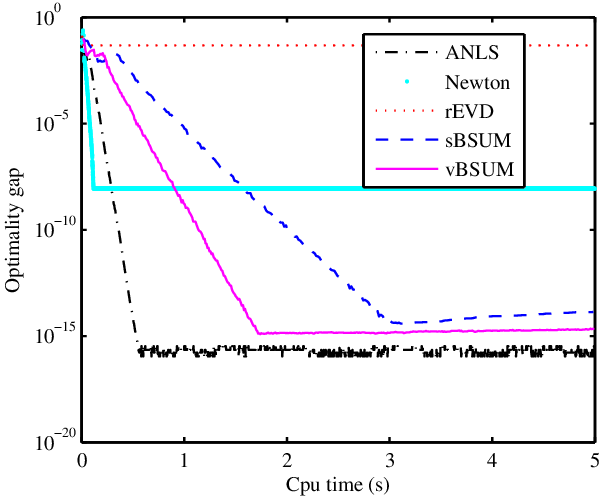}}
\label{fig:fig6}
\caption{The rEVD algorithm has fastest convergence but may not achieve a good approximation solution: $n=100$, $r=10$, the SGK method is used.}
\end{figure}

\begin{figure}[htbp]
\centering \subfigure[Objective function]{
\label{fig:subfig:fig4a} %% label for second subfigure
\includegraphics[width=0.5\textwidth]{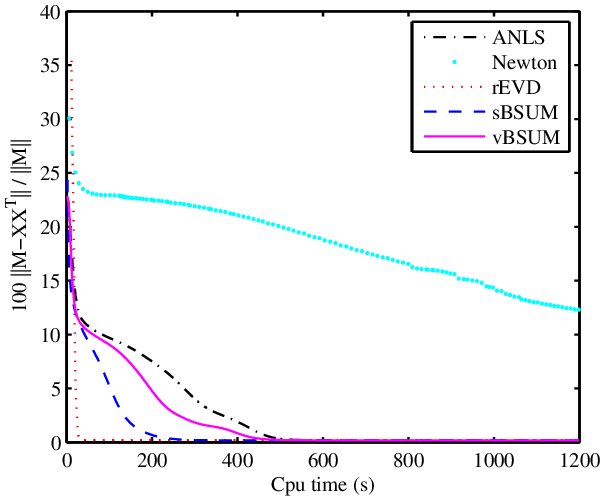}}
\hspace{-5pt} \subfigure[Optimality gap]{
\label{fig:subfig:fig4b} %% label for second subfigure
\includegraphics[width=0.5\textwidth]{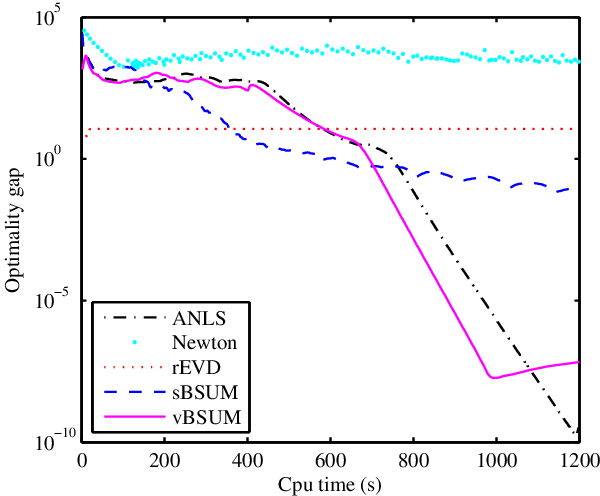}}
\label{fig:fig6}
\caption{The Newton algorithm converges very slowly for problems of large size: $n=2000$, $r=50$, $s=0.5$.}
\end{figure}

\begin{figure}[htbp]
\centering \subfigure[Objective function]{
\label{fig:subfig:fig4a} %% label for second subfigure
\includegraphics[width=0.5\textwidth]{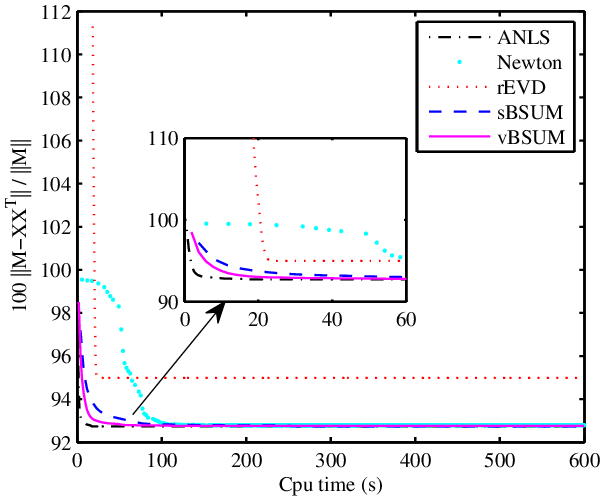}}
\hspace{-5pt} \subfigure[Optimality gap]{
\label{fig:subfig:fig4b} %% label for second subfigure
\includegraphics[width=0.5\textwidth]{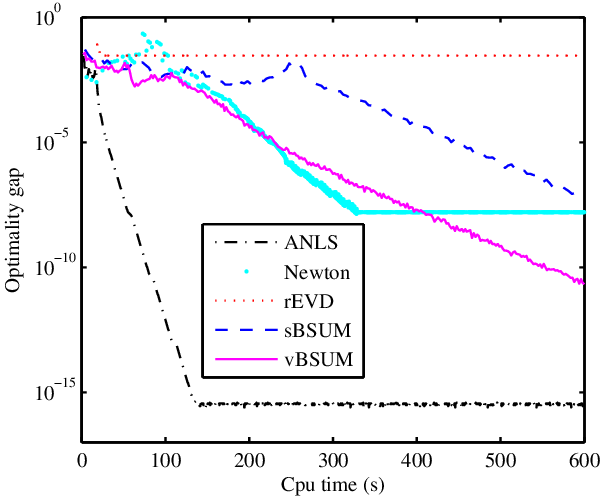}}
\label{fig:fig6}
\caption{The EVD operation in the rEVD algorithm may take a long time for problems of large size: $n=2000$, $r=50$, the SGK method is used.}
\end{figure}

\section{Conclusions}
In this paper, we have proposed both serial and parallel BSUM algorithms for the SNMF problem. All the algorithms are guaranteed to have convergence to the stationary solutions of the SNMF problem. The serial BSUM algorithms can perform competitively with state-of-the-art methods while the parallel BSUM algorithms can efficiently handle large-scale SNMF problems. We remark that our serial BSUM algorithms initialized from the final solution of the efficient rEVD algorithm \cite{Huang2014} could be expected as a good choice for solving small-scale SNMF problems. For large-scale SNMF problems, we suggest using either the randomized BSUM algorithms or parallel algorithms if a multi-core processor is available.

\appendices
\section{Derivation of the tuple $(a, b, c, d)$ in Eqs. (6-9)}%\vspace{-0.2in}%\vspace{-.2in}
Let $x_{ij}{=}x-\tilde{X}_{ij}$. Recall that $\bE_{ij}{\in} \Rdom^{n\times r}$/$\bE_{jj}{\in} \Rdom^{r\times r}$/$\bE_{ii}{\in}$ $\Rdom^{n\times n}$ is a matrix with $1$ in the $(i,j)$/$(j,j)$/$(i,i)$-th entry and $0$ elsewhere. Moreover, it holds true that $\bE_{ij}\bE_{ij}^T=\bE_{ii}$ and $\bE_{ji}\bE_{ji}^T=\bE_{jj}$. Then we have
\begin{align*}
&F(\bX)\\
&=\left\Vert\left(\tX+x_{ij}\bE_{ij}\right)\left(\tX+x_{ij}\bE_{ij}\right)^T-\bM\right\Vert^2\\
&=\left\Vert\tX\tX^T-\bM+x_{ij}\left(\tX\bE_{ij}^T+\bE_{ij}\tX^T\right)+x_{ij}^2\bE_{ii}\right\Vert^2%\\
\end{align*}
\begin{align*}
&=F(\tX){+}2x_{ij}\trace\left((\tX\tX^T{-}\bM)\left(\tX\bE_{ij}^T{+}\bE_{ij}\tX^T{+}x_{ij}\bE_{ii}\right)\right)\\
&~~~~+x_{ij}^2\Vert\tX\bE_{ij}^T+\bE_{ij}\tX^T+x_{ij}\bE_{ii}\Vert^2\\
& =x_{ij}^4+2\trace\left(\left(\tX\bE_{ij}^T+\bE_{ij}\tX^T\right)\bE_{ii}\right)x_{ij}^3\\
&~~~+\left(2\trace\left((\tX\tX^T-\bM)\bE_{ii}\right)+\left\Vert\tX\bE_{ij}^T+\bE_{ij}\tX^T\right\Vert^2\right)x_{ij}^2\\
&~~~+2\trace\left((\tX\tX^T-\bM)\left(\tX\bE_{ij}^T+\bE_{ij}\tX^T\right)\right)x_{ij}+f(\tX)\\
%&=\frac{a}{4}(x-\Xt_{ij})^4+\frac{b}{3}(x-\Xt_{ij})^3+\frac{c}{2}(x-\Xt_{ij})^2\\
%&~~~~~+d(x-\Xt_{ij})
\end{align*}
By comparing the above equality with Eqs. \eqref{eq:FX} and \eqref{eq:gx}), we have
\begin{align}
a &=4\nonumber\\
b&=6\trace\left(\left(\tX\bE_{ij}^T+\bE_{ij}\tX^T\right)\bE_{ii}\right)\nonumber\\
&=12\trace\left(\tX\bE_{ij}^T\right)\label{eq:beq}\\
&=12\Xt_{ij}\nonumber%\\%\left(\tX\bE_{ij}^T\right)_{ii}
\end{align}
\begin{align}
c&=2\left(2\trace\left((\tX\tX^T-\bM)\bE_{ii}\right)+\left\Vert\tX\bE_{ij}^T+\bE_{ij}\tX^T\right\Vert^2\right)\nonumber\\
&=4(\tX\tX^T-\bM)_{ii}+4\trace(\tX^T\tX\bE_{jj})+4\trace\left(\tX\bE_{ij}^T\tX\bE_{ij}^T\right)\nonumber\\
&=4\left((\tX\tX^T-\bM)_{ii}+(\tX^T\tX)_{jj}+\Xt_{ij}^2\right)\label{eq:ceq}\\
d&=2\trace\left((\tX\tX^T-\bM)\left(\tX\bE_{ij}^T+\bE_{ij}\tX^T\right)\right)\nonumber\\
&=4\trace\left((\tX\tX^T-\bM)\tX\bE_{ij}^T\right)\nonumber\\
&=4((\tX\tX^T-\bM)\tX)_{ij} \label{eq:deq}
\end{align}
where we have used the facts $\bE_{ij}^T\bE_{ii}=\bE_{ij}^T$, $\trace(\bA\bE_{ij}^T)=A_{ij}$, $\trace(\bA\bE_{ij}^T\bA\bE_{ij}^T)=A_{ij}^2$, and $\bM=\bM^T$ in (\ref{eq:beq}-\ref{eq:deq}).%\vspace{-15pt}

\section{Every nonzero stationary solution cannot be a local maximum}
\begin{theorem}
The proposed BSUM-type SNMF algorithms cannot get stuck in a local maximum.
\end{theorem}
\begin{proof}
Clearly the proposed BSUM-based SNMF algorithms can escape from the stationary point $\bX=\bm{0}$. Thus we consider any nonzero stationary solution $\bX$ below, i.e., we have $\Vert\bX\Vert^2> 0$.

Let $\by$ denote the vectorization of $\bX$. So we can express each column of $\bX$ as $\bX_{:i}=\bU_i\by$, where $\bU_i\in\Rdom^{n\times nr}$ is a null matrix except the $i$-th block being an $n\times n$ identity matrix. That is, we have
$\bX=[\bU_1\by~\bU_2\by~\ldots~\bU_r\by].$
It follows that
\begin{align}\label{eq:y=X}
\bX\bX^T&=[\bU_1\by~\bU_2\by~\ldots~\bU_r\by][\bU_1\by~\bU_2\by~\ldots~\bU_r\by]^T\nonumber\\
&=\sum_{i=1}^r \bU_i\by\by^T\bU_i^T
\end{align}
and
\begin{equation}\label{eq:obj}
F(\by)\equiv F(\bX)=\psi_1(\by)-\psi_2(\by)+\Vert\bM\Vert^2%\trace\left(\bX\bX^T\bX\bX^T\right)-2\trace\left(\bX\bX^T\bM\right)+\trace(\bM^2)$$
\end{equation}
where $\psi_1(\by)\triangleq\trace\left(\bX\bX^T\bX\bX^T\right)$ and $\psi_2(\by)\triangleq2\trace\left(\bX\bX^T\bM\right)$.

First, by the stationary condition $\nabla F(\by)^T(\bz-\by)\geq 0$, $\forall \bz\geq 0$ and the fact $\by\geq 0$, we obtain
$\nabla F(\by)\geq 0$.

Second, using Eq. \eqref{eq:y=X}, we have
\begin{align}
\psi_1(\by)&\triangleq\trace\left(\bX\bX^T\bX\bX^T\right)\nonumber\\
&=\trace\left(\left(\sum_{i=1}^r \bU_i\by\by^T\bU_i^T\right)^2\right)\nonumber\\
&=\trace\left(\sum_{i=1}^r\sum_{j=1}^r\bU_i\by\by^T\bU_i^T\bU_j\by\by^T\bU_j^T\right)\nonumber\\
&=\sum_{i=1}^r\sum_{j=1}^r\left(\by^T\bU_i^T\bU_j\by\right)\left(\by^T\bU_j^T\bU_i\by\right)
\end{align}%\vspace{-0.08in}
and
\begin{align}
\psi_2&(\by)\triangleq2\trace\left(\bX\bX^T\bM\right)\nonumber\\
&=2\trace\left([\bU_1\by~\bU_2\by~\ldots~\bU_r\by]^T\bM[\bU_1\by~\bU_2\by~\ldots~\bU_r\by]\right)\nonumber\\
&=2\sum_{i=1}^r\by^T\bU_i^T\bM\bU_i\by
\end{align}%\vspace{-0.08in}

Furthermore, evaluating the gradient of $\psi_1(\by)$ and $\psi_2(\by)$, we obtain
\begin{align}
\nabla&\psi_1(\by)=\sum_{i=1}^r\sum_{j=1}^r\left(\left(\by^T\bU_i^T\bU_j\by\right)\nabla\left(\by^T\bU_j^T\bU_i\by\right)\right.\nonumber\\
&~~~~~~~\left.+\left(\by^T\bU_j^T\bU_i\by\right)\nabla\left(\by^T\bU_i^T\bU_j\by\right)\right)\nonumber\\
&=\sum_{i=1}^r\sum_{j=1}^r\bigg(\left(\by^T\bU_i^T\bU_j\by\right)\left(\bU_i^T\bU_j+\bU_j^T\bU_i\right)\by\nonumber\\
&~~~~~~~~~~~~~+\left(\by^T\bU_j^T\bU_i\by\right)\left(\bU_i^T\bU_j+\bU_j^T\bU_i\right)\by\bigg)\label{eq:grad1}\\
&=\sum_{i=1}^r\sum_{j=1}^r\by^T\left(\bU_i^T\bU_j+\bU_j^T\bU_i\right)\by\left(\bU_i^T\bU_j+\bU_j^T\bU_i\right)\by\nonumber
\end{align}
where we have used the fact that $\nabla\left(\by^T\bA\by\right){=}(\bA{+}\bA^T)\by$, and%\vspace{-0.08in}
\begin{align}\label{eq:grad2}
\nabla\psi_2(\by)&=4\sum_{i=1}^r\bU_i^T\bM\bU_i\by
\end{align}
Further, the Hessain matrix of $\psi_1(\by)$ and $\psi_2(\by)$ are given by%\vspace{-0.08in}
\begin{align}
&\nabla^2\psi_1(\by)=\!\!\sum_{i=1}^r\sum_{j=1}^r\by^T\left(\bU_i^T\bU_j{+}\bU_j^T\bU_i\right)\by\left(\bU_i^T\bU_j{+}\bU_j^T\bU_i\right)\nonumber\\
&+2\!\sum_{i=1}^r\sum_{j=1}^r\left(\bU_i^T\bU_j{+}\bU_j^T\bU_i\right)\by\by^T\left(\bU_i^T\bU_j{+}\bU_j^T\bU_i\right)\label{eq:hessian1}\\
&\nabla^2\psi_2(\by)=4\sum_{i=1}^r\bU_i^T\bM\bU_i\label{eq:hessian2}
\end{align}

Using \eqref{eq:obj}, \eqref{eq:grad1}, \eqref{eq:grad2}, \eqref{eq:hessian1}, and \eqref{eq:hessian2}, we have
\begin{align}
\by^T&\nabla^2 f(\by)\by = \by^T\nabla^2\psi_1(\by)\by-\by^T\nabla^2\psi_2(\by))\by\nonumber\\
&=2\sum_{i=1}^r\sum_{j=1}^r\left(\by^T\left(\bU_i^T\bU_j+\bU_j^T\bU_i\right)\by\right)^2\nonumber\\
&~~~~~+\by^T\left(\nabla\psi_1(\by)-\nabla\psi_2(\by)\right)\nonumber\\
&\stackrel{(a)}{\geq} 2\sum_{i=1}^r\sum_{j=1}^r\left(\by^T\left(\bU_i^T\bU_j+\bU_j^T\bU_i\right)\by\right)^2\nonumber\\
&\geq 8\sum_{i=1}^r\left(\by^T\bU_i^T\bU_i\by\right)^2=8\Vert\bX\Vert^2>0\nonumber
\end{align}
where in (a) we have used the fact that $\by\geq 0$ and $\nabla F(\by)=\nabla\psi_1(\by)-\nabla\psi_2(\by)\geq 0$. This implies that the Hessian matrix $\nabla^2 F(\by)$ can only be either positive semidefinite (but not a null matrix) or indefinite. Therefore, the stationary point $\by$ cannot be a local maximum of $F$. Thus the proof is completed.
\end{proof}
\end{document}